\documentstyle[11pt,amssymb,amsmath,amscd,amsbsy,amsfonts,amsthm]{article}

\input xy
\xyoption{all}

\newtheorem{thm}{Theorem}[section]
\newtheorem{lem}[thm]{Lemma}
\newtheorem{prop}[thm]{Proposition}
\newtheorem{cor}[thm]{Corollary}
\newtheorem{rem}[thm]{Remark}
\newtheorem{dfn}[thm]{Definition}

\DeclareMathOperator{\Z}{\mathbb{Z}}
\DeclareMathOperator{\SH}{\mathcal{SH}}
\DeclareMathOperator{\DM}{\mathcal{DM}}
\DeclareMathOperator{\D}{\mathcal{D}}
\DeclareMathOperator{\X}{\mathfrak{X}}
\DeclareMathOperator{\A}{\mathcal{A}}
\DeclareMathOperator{\M}{\mathcal{M}}
\DeclareMathOperator{\G}{\mathcal{G}}
\DeclareMathOperator{\Pa}{\mathcal{P}}

\DeclareMathOperator{\hl}{holim}

\DeclareMathOperator{\mgl}{\mathrm{MGL}}
\DeclareMathOperator{\mbp}{\mathrm{MBP}}
\DeclareMathOperator{\bgl}{\mathrm{BGL}}
\DeclareMathOperator{\s}{\mathrm{S}}
\DeclareMathOperator{\com}{\mathbf{Comod}}
\DeclareMathOperator{\mo}{\mathbf{Mod}}
\DeclareMathOperator{\st}{\mathbf{Stable}}

\DeclareMathOperator{\hz}{\mathrm{H}{\mathbb Z}/2}
\DeclareMathOperator{\F}{{\mathbb F}_2}
\DeclareMathOperator{\Hom}{\mathrm{Hom}}
\DeclareMathOperator{\Ext}{\mathrm{Ext}}
\DeclareMathOperator{\Tor}{\mathrm{Tor}}

\title{\textsc{Cellular objects in isotropic motivic categories}}
\author{Fabio Tanania}
\date{}

\begin{document}

\maketitle

\begin{abstract}
Our main purpose is to describe the category of isotropic cellular spectra over flexible fields. Guided by \cite{GWX}, we show that it is equivalent, as a stable $\infty$-category equipped with a $t$-structure, to the derived category of left comodules over the dual of the classical topological Steenrod algebra. In order to obtain this result, the category of isotropic cellular modules over the motivic Brown-Peterson spectrum is also studied, and isotropic Adams and Adams-Novikov spectral sequences are developed. As a consequence, we also compute hom sets in the category of isotropic Tate motives between motives of isotropic cellular spectra. 
\end{abstract}

\section{Introduction}

Isotropic categories are local versions of motivic categories, obtained by, roughly speaking, killing all anisotropic varieties. Although they often have a handier structure than their global versions, they exhibit some key characteristics of both motivic and classical topological phenomena. In \cite{V}, Vishik introduced the isotropic triangulated category of motives and computed the isotropic motivic cohomology of the point, which is strongly related to the Milnor subalgebra. By following this lead, we studied in \cite{T} the isotropic stable motivic homotopy category. In particular, we identified the isotropic motivic homotopy groups of the sphere spectrum with the cohomology of the topological Steenrod algebra, i.e. the $E_2$-page of the classical Adams spectral sequence. These results are quite surprising since they show that topological objects naturally arise from isotropic environments, which could lead to a fruitful exchange between topology and isotropic motivic theory.

Motivic categories, constructed by Morel and Voevodsky (see \cite{MV} and \cite{V3}) in order to study algebraic varieties by topological means, are extremely rich categories. Even over an algebraically closed field, they are more complex than the respective topological counterparts. For example, while every object in the classical stable homotopy category is cellular, i.e. built up by attaching spheres, not every motivic spectrum is cellular, since many algebro-geometric phenomena come into the picture. In spite of this, it is still interesting to understand the structure of the category of cellular objects in motivic stable homotopy theory. This project was initiated by Dugger and Isaksen in \cite{DI2} and much attention has been dedicated to it since then. Our work, in particular, is concerned with understanding the structure of the subcategories of cellular objects in isotropic categories, which we believe could shed light on the deep interconnection with topology.

We have already highlighted that motivic categories are particularly challenging to study. For example, one of the difficulties that one does not encounter in classical topology is the presence of an object $\tau$ that appears in various incarnations throughout motivic homotopy theory, sometimes as an element of the motivic cohomology of the ground field and sometimes as a map in the $2$-complete motivic stable homotopy groups of spheres. Hence, the principal task is to find first some substitutes of the original motivic categories and tools which could help in the process of analysing them. In the case of algebraically closed fields, for example, topological realisation is a very helpful tool, since it allows to study the initial motivic category by looking at its deformation $\tau=1$, which happens to be just the classical stable homotopy theory (see \cite{DI}). However, in this process part of the information is lost and one could try to recover it by studying other deformations, for example $\tau=0$. This was done by Isaksen in \cite{I}, Gheorghe in \cite{G} and Gheorghe-Wang-Xu in \cite{GWX}. More precisely, in \cite{I} the stable motivic homotopy groups of $C{\tau}$, i.e. the cofiber of $\tau$, are identified with the $E_2$-page of the classical Adams-Novikov spectral sequence, while in \cite{G} the motivic spectrum $C{\tau}$ is provided with an $E_{\infty}$-ring structure inducing an isomorphism of rings with higher products between $\pi_{**}(C\tau)$ and the classical Adams-Novikov $E_2$-page. A parallel result for isotropic categories was obtained in \cite{T}, where the isotropic sphere spectrum $\X$ has been equipped with an $E_{\infty}$-ring structure inducing an isomorphism of rings with higher products between $\pi_{**}(\X)$ and the classical Adams $E_2$-page. Moreover, in \cite{GWX} the category of $C{\tau}$-cellular spectra is described, which is proved to be equivalent as a stable $\infty$-category equipped with a $t$-structure (see \cite{Lu}) to the derived category of left $\mathrm{BP}_*\mathrm{BP}$-comodules concentrated in even degrees, where $\mathrm{BP}$ is the Brown-Peterson spectrum and $\mathrm{BP}_*\mathrm{BP}$ its $\mathrm{BP}$-homology.

In this work, we intend to follow a similar path for isotropic categories. Recall that a field $k$ is called flexible if it is a purely transcendental extension of countable infinite degree over some other field. In our situation it is really essential to work over flexible fields since, as highlighted in \cite{V}, these are the ground fields over which the isotropic categories behave particularly well. For example, over algebraically closed fields, due to the lack of anisotropic varieties, the isotropic category would be just the same as the original motivic category, so in this case the isotropic localization produces nothing new. We are encouraged by the evident parallel between the computations of $\pi_{**}(C\tau)$ over complex numbers (see \cite{I} and \cite{G}), on the one hand, and of $\pi_{**}(\X)$ over flexible fields (see \cite{T}), on the other. More precisely, we have been guided by the idea that studying the isotropic stable motivic homotopy category over a flexible field is similar in some sense to studying the stable $\infty$-category of $C{\tau}$-cellular spectra in the motivic stable homotopy category over complex numbers. Indeed, they obviously share some common features which is highlighted by the following theorem that is the main result of this paper.

\begin{thm}
	Let $k$ be a flexible field of characteristic different from $2$. Then, there exists a $t$-exact equivalence of stable $\infty$-categories
	$$\D^b(\A_{*}-\com_{*}) \xrightarrow{\cong} \X-\mo^b_{cell,\hz}$$
	where $\A_{*}$ is the classical dual Steenrod algebra and $\X-\mo^b_{cell,\hz}$ is the stable $\infty$-category of $\hz$-complete $\X$-cellular modules having $\mbp$-homology non-trivial in only finitely many Chow-Novikov degrees (the superscript ``b" stands for ``bounded", see Definition \ref{dc}).
\end{thm}

As a consequence, we obtain that the category of isotropic cellular spectra is completely algebraic, which makes it easier to study. Moreover, it is deeply related to classical topology, as foreseeable from results in \cite{V} and \cite{T}.

In order to achieve our main results, we need several tools. In particular, it is necessary to develop and study isotropic versions of both the Adams spectral sequence and the Adams-Novikov spectral sequence. This requires a focus on the motivic Brown-Peterson spectrum $\mbp$ (see \cite{Ve}) from an isotropic point of view. In particular, we note that the isotropic Brown-Peterson spectrum is an $E_{\infty}$-ring spectrum, in contrast to the topological picture where $\mathrm{BP}$ has shown not to admit an $E_{\infty}$-ring structure by Lawson in \cite{La}. Then, we use techniques developed by Gheorghe-Wang-Xu in \cite{GWX}, based on Lurie's results (see \cite{Lu}), to, first, describe in algebraic terms the category of isotropic $\mbp$-cellular modules and, then, the category of all isotropic cellular spectra. At the end, we are also able to provide some results about the cellular subcategory of the isotropic triangulated category of motives, i.e. the category of isotropic Tate motives.\\

\textbf{Outline.} We now briefly present the contents of each section of this paper. In Section 2, we provide the main notations that are followed throughout this work. Then, we move on to Section 3 by recalling isotropic categories and their main properties, mostly referring to results in \cite{V} and \cite{T}. Since we are mainly interested in cellular objects, we recall in Section 4 definitions and some of the main results from \cite{DI2}, which are useful in the rest of the paper. Section 5 is devoted to a deep analysis of the isotropic motivic Adams spectral sequence, which was already initiated in \cite{T}. These results are used in Section 6 to study the motivic Brown-Peterson spectrum from an isotropic perspective. In particular, we compute its isotropic stable homotopy groups. Sections 7 and 8 are modeled on Sections 3, 4 and 5 of \cite{GWX}. More precisely, in Section 7 we endow the isotropic motivic Brown-Peterson spectrum with an $E_{\infty}$-ring structure and, then, identify as a triangulated category the category of isotropic $\mbp$-cellular spectra with the category of bigraded $\F$-vector spaces. In Section 8, after developing an isotropic Adams-Novikov spectral sequence, we describe the category of isotropic cellular spectra in algebraic terms as the derived category of comodules over the dual of the Steenrod algebra equipped with a $t$-structure. Finally, in Section 9, we provide an algebraic description of the hom-sets in the category of isotropic motives between motives of isotropic cellular spectra, which is a step forward in the understanding of the category of isotropic Tate motives.\\

\textbf{Acknowledgements.} I would like to thank Alexander Vishik for very helpful comments and Dan Isaksen for having pointed out to me the work by Gheorghe-Wang-Xu on which this paper is modeled. I am extremely grateful to Tom Bachmann for very useful remarks. I also wish to thank the referees for very useful comments which helped to improve the exposition and to simplify Section 7.

\section{Notation}

Let us start by fixing some notations we use throughout this paper.\\

\begin{tabular}{c|c}
	$k$ & flexible field with $char(k) \neq 2$\\
	$\SH(k)$ & stable motivic homotopy category over $k$\\
	$\SH(k/k)$ & isotropic stable motivic homotopy category over $k$\\
	$\DM(k)$ & triangulated category of motives with $\Z/2$-coefficients over $k$\\
	$\DM(k/k)$ & isotropic triangulated category of motives with $\Z/2$-coefficients over $k$\\
	$\pi_{**}(-)$ & stable motivic homotopy groups\\
	$\pi^{iso}_{**}(-)$ & isotropic stable motivic homotopy groups\\
	$H_{**}(-),H^{**}(-)$ & motivic homology and cohomology with $\Z/2$-coefficients\\
	$H^{iso}_{**}(-),H_{iso}^{**}(-)$ & isotropic motivic homology and cohomology with $\Z/2$-coefficients\\
	$H_{**}(k),H^{**}(k)$ & motivic homology and cohomology with $\Z/2$-coefficients of $Spec(k)$\\
	$H_{**}(k/k),H^{**}(k/k)$ & isotropic motivic homology and cohomology with $\Z/2$-coefficients of $Spec(k)$\\
	$\A^{**}(k),\A_{**}(k)$ & mod 2 motivic Steenrod algebra and its dual\\
	$\A^{**}(k/k),\A_{**}(k/k)$ & mod 2 isotropic motivic Steenrod algebra and its dual\\
	$\A^{*},\A_{*}$ & mod 2 topological Steenrod algebra and its dual\\
	$\G^{**},\G_{**}$ & bigraded mod 2 topological Steenrod algebra and its dual\\
	& i.e. $\G^{2q,q}=\A^q$, $\G^{p,q}=0$ for $p \neq 2q$ and similarly for the dual\\
	$\M^{**}$ & Milnor subalgebra $\Lambda_{\F}(Q_i)_{i \geq 0}$ of $\A^{**}(k/k)$\\
	& where $Q_i$ are the Milnor operations in bidegrees $(2^i-1)[2^{i+1}-1]$\\
	$\s$ & motivic sphere spectrum\\
	$\hz$ & motivic Eilenberg-MacLane spectrum with $\Z/2$-coefficients\\
	$\mgl$ & motivic algebraic cobordism spectrum\\
	$\mbp$ & motivic Brown-Peterson spectrum at the prime $2$\\
	$\X$ & isotropic sphere spectrum
\end{tabular}\\

We denote hom-sets in $\SH(k)$ by $[-,-]$ and the suspension $\s^{p,q} \wedge X$ of a motivic spectrum $X$ by $\Sigma^{p,q}X$. Moreover, if $E$ is a motivic $E_{\infty}$-ring spectrum, the stable $\infty$-category of $E$-modules (see \cite{Lu}) is denoted by $E-\mo$, its smash product by $- \wedge_E -$ and hom-sets in its homotopy category by $[-,-]_E$.

If $R$ is an algebra and $C$ a coalgebra, then we denote by $R-\mo$ and $C-\com$ the categories of left $R$-modules and left $C$-comodules respectively. Hom-sets in these categories are both denoted by $\Hom _R(-,-)$ and $\Hom _C(-,-)$ and it will be clear from the context if they are meant to be hom of modules or comodules. For a bigraded object $M_{**}$ (respectively $M^{**}$) we also denote by $\Sigma^{p,q}M_{**}$ (respectively $\Sigma^{p,q}M^{**}$) its suspension, i.e. the bigraded object defined by $\Sigma^{p,q}M_{a,b}=M_{a-p,b-q}$ (respectively $\Sigma^{p,q}M^{a,b}=M^{a+p,b+q}$). The convention for bigraded homomorphisms between bigraded objects is the following:
$$\Hom ^{p,q}(M_{**},N_{**})=\Hom ^{0,0}(\Sigma^{p,q}M_{**},N_{**})$$
and
$$\Hom ^{p,q}(M^{**},N^{**})=\Hom ^{0,0}(\Sigma^{p,q}M^{**},N^{**})$$
where $\Hom ^{0,0}(-,-)$ denotes the bidegree preserving homomorphisms. Moreover, the bounded derived categories of $R-\mo$ and $C-\com$ are denoted by $\D^b(R-\mo)$ and $\D^b(C-\com)$ respectively. 

\section{Isotropic motivic categories}

In this section we want to introduce the main categories we consider throughout this paper, namely isotropic motivic categories. These categories are built from the respective motivic ones by, roughly speaking, killing all anisotropic varieties. We refer to \cite[Section 2]{V} and \cite[Section 2]{T} for more details on the construction and properties of isotropic categories.

Let us recall first the definition of flexible field from \cite{V}.

\begin{dfn}
	A field $k$ is called flexible if it is a purely transcendental extension of countable infinite degree, i.e. $k=k_0(t_1,t_2,\dots)$ for some other field $k_0$. 
\end{dfn}

Once and for all we consider a flexible base field $k$ of characteristic different from 2. We proceed by recalling the definition of a fundamental object in $\SH(k)$ for the construction of the isotropic stable motivic homotopy category $\SH(k/k)$. 

\begin{dfn}
	\normalfont
Denote by $Q$ the disjoint union of all connected anisotropic (mod 2) varieties over $k$, i.e. varieties which do not have closed points of odd degree, and by $\check{C}(Q)$ its \v{C}ech simplicial scheme, i.e. $\check{C}(Q)_n=Q^{n+1}$ with face and degeneracy maps given respectively by partial projections and partial diagonals. We define the isotropic sphere spectrum $\X$ as $Cone(\Sigma^{\infty}_+\check{C}(Q) \rightarrow \s)$ in $\SH(k)$ .
\end{dfn} 

We recall from \cite[Section 2]{T} that $\X$ is an idempotent monoid, i.e. there is an equivalence $\X \wedge \X \cong \X$ induced by the map $\s \rightarrow \X$, and so an $E_{\infty}$-ring spectrum (see \cite[Proposition 6.1]{T}).

\begin{dfn}
	\normalfont
The full triangulated subcategory $\X \wedge \SH(k)$ of $\SH(k)$ will be called the isotropic stable motivic homotopy category, and denoted by $\SH(k/k)$.
\end{dfn} 

This triangulated category has very nice properties, in particular it is both localising and colocalising (see \cite[Section 2]{T}). The very same construction was done first for $\DM(k)$ by Vishik in \cite{V}, by tensoring the triangulated category of motives with the idempotent ${\mathrm M}(\X)$, where ${\mathrm M}:\SH(k) \rightarrow \DM(k)$ is the motivic functor. 

\begin{dfn}
	\normalfont
The full triangulated subcategory ${\mathrm M}(\X) \otimes \DM(k)$ of $\DM(k)$ will be called the isotropic category of motives, and denoted by $\DM(k/k)$.
\end{dfn}

The following result tells us that the isotropic stable motivic homotopy category is nothing else but the stable $\infty$-category of $\X$-modules.

\begin{prop}
	The isotropic stable motivic homotopy category $\SH(k/k)$ is equivalent to the stable $\infty$-category $\X-\mo$ of modules over the motivic $E_{\infty}$-ring spectrum $\X$.
\end{prop}
\begin{proof}
	It follows immediately from \cite[Proposition 4.8.2.10]{Lu}.
\end{proof}

\begin{rem}\label{bach}
	\normalfont
Since by construction $\X$ kills all anisotropic varieties, it kills in particular non-trivial quadratic extensions. Consider an element $x$ in $k$ such that neither $x$ nor $-x$ is a square. Then, we have that $\X \wedge \Sigma^{\infty}_+Spec(k(\sqrt{x}))$ and $\X \wedge \Sigma^{\infty}_+Spec(k(\sqrt{-x}))$ are both zero. This implies that the Euler characteristics of $Spec(k(\sqrt{x}))$ and $Spec(k(\sqrt{-x}))$, which are respectively equal to $\langle 2 \rangle (1 + \langle x\rangle)$ and $\langle 2 \rangle (1 + \langle -x\rangle)$ in $\pi_{0,0}(\s) \cong {\mathrm GW}(k)$ (see \cite[Corollary 11.2]{Le} and \cite[Theorem 6.2.2]{Mo}), vanish in $\pi_{0,0}(\X)$. It follows that $1 + \langle x\rangle$ and $1 + \langle -x\rangle$ vanish in $\pi_{0,0}(\X)$ and so does their sum
$$2 + \langle x\rangle + \langle -x\rangle= 2 + \langle 1\rangle + \langle -1\rangle=3 +\langle -1\rangle.$$
Hence, we have that $-3=\langle -1\rangle$, and so $9=1$, i.e. $8=0$ in $\pi_{0,0}(\X)$. From all this one deduces that $\X$ is $2$-power torsion.\footnote{I am grateful to Tom Bachmann for this argument.}
\end{rem}

We are now ready to define isotropic motivic homotopy groups and isotropic motivic homology and cohomology. 

\begin{dfn}
	\normalfont
Let $X$ be a motivic spectrum in $\SH(k)$. Then, the isotropic stable motivic homotopy groups of $X$ are defined by
$$\pi_{**}^{iso}(X)=[\s^{**},\X \wedge X]=\pi_{**}(\X \wedge X).$$
\end{dfn}

Recall that motivic cohomology with $\Z/2$-coefficients is represented by the motivic Eilenberg-MacLane spectrum $\hz$. Then, we define isotropic motivic cohomology as the cohomology theory represented by the motivic $E_{\infty}$-ring spectrum $\X \wedge \hz$. 

\begin{dfn}
	\normalfont
For any $X$ in $\SH(k)$, we define the isotropic motivic cohomology of $X$ as
$$H^{**}_{iso}(X)=[X,\Sigma^{**}(\X \wedge \hz)]$$
and the isotropic motivic homology of $X$ as
$$H_{**}^{iso}(X)=[\s^{**},\X \wedge \hz \wedge X] = H_{**}(\X \wedge X).$$
\end{dfn}

The isotropic motivic cohomology of the point was computed by Vishik in \cite{V}. We report the result in the next theorem.

\begin{thm}\label{vis}
	Let $k$ be a flexible field. Then, for any $i \geq 0$ there exists a unique cohomology class $r_i$ of bidegree $(-2^i+1)[-2^{i+1}+1]$ such that
	$$H^{**}(k/k) \cong \Lambda_{\F}(r_i)_{i \geq 0}$$
	and $Q_j r_i=\delta_{ij}$, where $Q_j$ are the Milnor operations.
\end{thm}
\begin{proof}
	See \cite[Theorem 3.7]{V}. 
\end{proof}

At this point, we want to introduce the isotropic motivic Steenrod algebra $\A^{**}(k/k)$ and its dual $\A_{**}(k/k)$. They are defined respectively as the isotropic motivic cohomology and homology of the motivic Eilenberg-MacLane spectrum. 

\begin{dfn}
\normalfont
The isotropic motivic Steenrod algebra is defined by
$$\A^{**}(k/k)=H^{**}_{iso}(\hz)=[\hz,\Sigma^{**}(\X \wedge \hz)] \cong [\X \wedge \hz,\Sigma^{**}(\X \wedge \hz)]$$
and its dual by
$$\A_{**}(k/k)=H_{**}^{iso}(\hz)=[\s^{**},\X \wedge \hz \wedge \hz].$$
\end{dfn}

The structure of $\A^{**}(k/k)$ was studied in \cite[Section 3]{T}. We summarise the main results in the next proposition.

\begin{prop}\label{st}
	Let $k$ be a flexible field. Then, there exists an isomorphism of $H^{**}(k/k) - \M^{**}$-bimodules
	$$\A^{**}(k/k) \cong H^{**}(k/k) \otimes_{\F} \G^{**} \otimes_{\F} \M^{**}$$
		where $\M^{**}$ is the Milnor subalgebra $\Lambda_{\F}(Q_i)_{i \geq 0}$ and $\G^{**}$ is the bigraded topological Steenrod algebra, i.e. $\G^{2n,n}=\A^n$.
\end{prop}
\begin{proof}
See \cite[Propositions 3.5, 3.6 and 3.7]{T}. 
\end{proof}

By projecting the motivic Cartan formulas (see \cite[Propositions 9.7 and 13.4]{V2}) to the isotropic category one gets a coproduct on $\A^{**}(k/k)$ given by:
$$\Delta(Sq^{2n})=\sum_{i+j=n}Sq^{2i}\otimes Sq^{2j};$$
$$\Delta(Q_i)=Q_i\otimes 1 + 1\otimes Q_i.$$
This coproduct structures $\A^{**}(k/k)$ as a coalgebra whose dual is described as an $H_{**}(k/k)$-algebra by
$$\A_{**}(k/k) \cong \frac {H_{**}(k/k)[\tau_i,\xi_j]_{i \geq 0,j \geq 1}} {(\tau_i^2)}$$
where $\tau_i$ is the dual of the Milnor operation $Q_i$ and $\xi_j$ is the dual of the motivic cohomology operation $Sq^{2^j}\cdots Sq^2$. The coproduct in $\A_{**}(k/k)$ is given by (see \cite[Lemma 12.11]{V2}):
$$\psi(\xi_k)=\sum_{i=0}^k \xi_{k-i}^{2^i} \otimes \xi_i;$$
$$\psi(\tau_k)=\sum_{i=0}^k \xi_{k-i}^{2^i} \otimes \tau_i+\tau_k\otimes 1.$$

\begin{rem}
	\normalfont
By Proposition \ref{st}, the projection from $\A^{**}(k/k)$ to its quotient by the left ideal generated by Milnor operations provides a homomorphism
$$\A^{**}(k/k) \rightarrow H^{**}(k/k) \otimes_{\F} \G^{**}.$$
This map induces a left $\A^{**}(k/k)$-action on $H^{**}(k/k) \otimes_{\F} \G^{**}$ and, dually, a left $\A_{**}(k/k)$-coaction on $H_{**}(k/k) \otimes_{\F} \G_{**}$, where $\G_{**}$ is the subalgebra $\F[\xi_1,\xi_2,\dots]$. 
\end{rem}

\section{Cellular motivic spectra}

In this work, we are mostly interested in cellular objects of isotropic motivic categories. We recall from \cite[Remark 7.4]{DI2} that the category of cellular motivic spectra, which we denote by $\SH(k)_{cell}$, is the localising subcategory of $\SH(k)$ generated by the spheres $\Sigma^{p,q}\s$. Similarly, the category of Tate motives, which we denote by $\DM(k)_{Tate}$, is the localising subcategory of $\DM(k)$ generated by the Tate motives $T(q)[p]$. If $E$ is a motivic $E_{\infty}$-ring spectrum, then we denote by $E-\mo_{cell}$ the stable $\infty$-category of $E$-cellular modules, i.e. the localising subcategory of $E-\mo$ generated by $\Sigma^{p,q}E$. 

\begin{dfn}
	\normalfont
	The category of $\X$-cellular modules will be called the category of isotropic cellular motivic spectra, and denoted by $\SH(k/k)_{cell}$. In the same way, the full localising subcategory of $\DM(k/k)$ generated by the objects ${\mathrm M}(\X)(q)[p]$ will be called the category of isotropic Tate motives, and denoted by $\DM(k/k)_{Tate}$.  
\end{dfn}

A fundamental property of the category of cellular objects is that isomorphisms can be detected by motivic homotopy groups, as reported in the following result.

\begin{prop}\label{check}
Let $E$ be a motivic $E_{\infty}$-ring spectrum and  $X \rightarrow Y$ be a map of $E$-cellular motivic spectra that induces isomorphisms on $\pi_{p,q}$ for all $p $ and $q$ in $\Z$. Then, the map is a weak equivalence.
\end{prop}
\begin{proof}
See \cite[Corollary 7.2 and Section 7.9]{DI2}. 
\end{proof}

Another essential advantage of dealing with cellular objects is that they allow the construction of very useful convergent spectral sequences.

\begin{prop}\label{ss}
	Let $E$ be a motivic $E_{\infty}$-ring spectrum and $N$ a left $E$-module. If $M$ is a right $E$-cellular spectrum, then there is a strongly convergent spectral sequence
	$$E^2_{s,t,u} \cong \Tor^{\pi_{**}(E)}_{s,t,u}(\pi_{**}(M),\pi_{**}(N))\Longrightarrow \pi_{s+t,u}(M \wedge_E N).$$
	If $M$ is a left $E$-cellular motivic spectrum, then there is a conditionally convergent spectral sequence
	$$E_2^{s,t,u} \cong \Ext_{\pi_{**}(E)}^{s,t,u}(\pi_{**}(M),\pi_{**}(N)) \Longrightarrow [\Sigma^{t-s,u}M,N]_E.$$
\end{prop}
\begin{proof}
	See \cite[Propositions 7.7 and 7.10]{DI2}. 
\end{proof}

\section{The isotropic motivic Adams spectral sequence}

In this section we recall the construction of the isotropic motivic Adams spectral sequence (see \cite[Section 4]{T}). Moreover, we study the circumstances under which the $E_2$-page is expressible in terms of $\Ext$-groups over the isotropic motivic Steenrod algebra.

\begin{dfn}
	\normalfont
Let $Y$ be an isotropic motivic spectrum, i.e. an object in $\X-\mo$. Then, the standard isotropic motivic Adams resolution of $Y$ consists of the Postnikov system
$$
\xymatrix{
	\dots \ar@{->}[r] &  (\overline{\X \wedge \hz})^{\wedge s} \wedge Y \ar@{->}[r] \ar@{->}[d] &  \dots \ar@{->}[r]   & \overline{\X \wedge \hz} \wedge Y \ar@{->}[r] \ar@{->}[d]	 & Y \ar@{->}[d]  \\
	&	\X \wedge \hz\wedge (\overline{\X \wedge \hz})^{\wedge s} \wedge Y \ar@{->}[ul]^{[1]} & & \X \wedge \hz \wedge \overline{\X \wedge \hz} \wedge Y \ar@{->}[ul]^{[1]}  &	\X \wedge \hz \wedge Y \ar@{->}[ul]^{[1]} 
}
$$
where $\overline{\X \wedge \hz}$ is defined by the following exact triangle in $\SH(k)$:
$$\overline{\X \wedge \hz} \rightarrow \s \rightarrow \X \wedge \hz \rightarrow \Sigma^{1,0}\overline{\X \wedge \hz}.$$
By applying motivic homotopy groups functors $\pi_{**}$ to the previous Postnikov system we get an unrolled exact couple, which induces in turn a spectral sequence with $E_1$-page described by
$$E_1^{s,t,u}\cong \pi_{t-s,u}(\X \wedge \hz\wedge (\overline{\X \wedge \hz})^{\wedge s} \wedge Y)$$
and first differential
$$d_1^{s,t,u}:\pi_{t-s,u}(\X \wedge \hz\wedge (\overline{\X \wedge \hz})^{\wedge s} \wedge Y) \rightarrow \pi_{t-s-1,u}(\X \wedge \hz\wedge (\overline{\X \wedge \hz})^{\wedge s+1} \wedge Y).$$
In general, differentials on the $E_r$-page have tri-degrees given by
$$d_r^{s,t,u}:E_r^{s,t,u} \rightarrow E_r^{s+r,t+r-1,u}.$$
We call this spectral sequence isotropic motivic Adams spectral sequence. 
\end{dfn}

The isotropic Adams spectral sequence converges to the homotopy groups of a motivic spectrum closely related to $Y$, namely its $\X \wedge \hz$-nilpotent completion that we denote by $Y^{\wedge}_{\X \wedge \hz}$. Before proceeding, let us recall from \cite[Section 5]{Bo} how to construct the $E$-nilpotent completion of a spectrum $Y$ for a homotopy ring spectrum $E$. 

\begin{dfn}
	\normalfont
Let $E$ be a homotopy ring spectrum and $Y$ a motivic spectrum in $\SH(k)$. First, define $\overline{E}$ by the following distinguished triangle in $\SH(k)$:
$$\overline{E} \rightarrow \s \rightarrow E \rightarrow \Sigma^{1,0}\overline{E}.$$
Then, define $\overline{E}_{n}$ as $Cone(\overline{E}^{\wedge n+1} \rightarrow \s)$ in $\SH(k)$. This way one gets an inverse system
$$ \dots \rightarrow \overline{E}_{n} \wedge Y \rightarrow \dots \rightarrow \overline{E}_{1} \wedge Y \rightarrow \overline{E}_{0} \wedge Y$$
and the $E$-nilpotent completion of $Y$ is the motivic spectrum defined by $Y^{\wedge}_E=\hl (\overline{E}_{n} \wedge Y)$.
\end{dfn}

 Note that, by \cite[Proposition 2.3]{T}, if $Y$ is an isotropic motivic spectrum then also $Y^{\wedge}_E$ is so.

\begin{prop}\label{conv}
	Let $Y$ be an isotropic motivic spectrum. If $\varprojlim_{r}^1 E_r^{s,t,u}=0$ for any $s,t,u$, then the isotropic motivic Adams spectral sequence for $Y$ is strongly convergent to the stable motivic homotopy groups of the $\hz$-nilpotent completion of $Y$.	
\end{prop}
\begin{proof}
	By \cite[Proposition 6.3]{Bo} and \cite[Remark 6.11]{DI}, under the vanishing hypothesis on $\varprojlim_{r}^1 E_r^{s,t,u}$, the isotropic motivic Adams spectral sequence strongly converges to $\pi_{**}(Y^{\wedge}_{\X \wedge \hz})$. It only remains to notice that, since $Y$ is a $\X$-module, its $\hz$-nilpotent and $\X \wedge \hz$-nilpotent completions coincide. In fact, after smashing with $\X$ the morphism of distinguished triangles
	$$
	\xymatrix{
		\overline{\hz} \ar@{->}[r] \ar@{->}[d]&  \s \ar@{->}[r] \ar@{=}[d] &  \hz \ar@{->}[r] \ar@{->}[d]  & \Sigma^{1,0}\overline{
			\hz}   \ar@{->}[d]\\
		\overline{\X \wedge \hz} \ar@{->}[r] &  \s \ar@{->}[r]  &  \X \wedge\hz \ar@{->}[r]  & \Sigma^{1,0}\overline{
		\X \wedge	\hz}  
	}
	$$
	one gets 
	$$
	\xymatrix{
	\X \wedge	\overline{\hz} \ar@{->}[r] \ar@{->}[d]&  \X  \ar@{->}[r] \ar@{=}[d] & \X \wedge \hz \ar@{->}[r] \ar@{->}[d]^{\cong}  & \Sigma^{1,0}\X \wedge\overline{
			\hz}   \ar@{->}[d]\\
	\X \wedge	\overline{\X \wedge \hz} \ar@{->}[r] &  \X \ar@{->}[r]  &  \X \wedge \X \wedge\hz \ar@{->}[r]  & \Sigma^{1,0}\X \wedge \overline{
			\X \wedge	\hz}  
	}
	$$
	since $\X$ is an idempotent in $\SH(k)$. It follows that $\X \wedge	\overline{\hz} \cong \X \wedge	\overline{\X \wedge \hz}$, and so $\X \wedge \overline{\hz}_n \cong \X \wedge (\overline{\X \wedge \hz})_n$ for any $n$. Therefore, since $Y \cong \X \wedge Y$ one obtains that
	\begin{align*}
		Y^{\wedge}_{\X \wedge \hz}&= \hl ((\overline{\X \wedge \hz})_n \wedge Y)\cong \hl (\X \wedge (\overline{\X \wedge \hz})_n \wedge Y) \\
	&\cong \hl (\X \wedge \overline{\hz}_n \wedge Y) \cong \hl  (\overline{\hz}_n \wedge Y)=Y^{\wedge}_{ \hz}
	\end{align*}
	which is what we wanted to show.
\end{proof}

\begin{rem}\label{abc}
	\normalfont
By \cite[Section 5.2 and Theorem 1.0.3]{Ma}, the $\hz$-completion of a connective motivic spectrum coincides with its $(2,\eta)$-completion. Since all isotropic motivic spectra are $2$-power torsion (see Remark \ref{bach}), and so $2$-complete, the previous result establishes the convergence of the isotropic Adams spectral sequence for a connective isotropic spectrum to the motivic stable homotopy groups of its $\eta$-completion.	
\end{rem}

\begin{dfn}
	\normalfont
A spectral sequence $\{E_r^{s,t,u}\}$ is called Mittag-Leffler if for each $s,t,u$ there exists $r_0$ such that $E_r^{s,t,u} \cong E_{\infty}^{s,t,u}$ whenever $r > r_0$.
\end{dfn} 

Note that every Mittag-Leffler spectral sequence satisfies the condition $\varprojlim_{r}^1 E_r^{s,t,u}=0$ for any $s,t,u$ (see \cite[after Proposition 6.3]{Bo}). We will see that in many important cases the isotropic Adams spectral sequence is Mittag-Leffler, which guarantees strong convergence.

Now, we would like to understand what conditions we need to impose on $Y$ in order to be able to express the $E_2$-page of the isotropic Adams spectral sequence in terms of $\Ext$-groups over the isotropic motivic Steenrod algebra. First, we need the following lemmas.

\begin{lem}\label{kun}
	Let $k$ be a flexible field and $Y$ an object in $\X-\mo$. Then, there exists an isomorphism of left $H_{**}(k/k)$-modules
	$$H^{iso}_{**}(\X \wedge \hz \wedge Y) \cong \A_{**}(k/k) \otimes_{H_{**}(k/k)} H^{iso}_{**}(Y).$$
\end{lem}
\begin{proof}
Since by \cite[Theorem 5.10]{H} $\hz \wedge \hz$ is a split $\hz$-module, i.e. it is equivalent to a wedge sum of the form $\bigvee_{\alpha \in A}\Sigma^{p_{\alpha},q_{\alpha}}\hz$, we have that
	\begin{align*}
		\A_{**}(k/k)& \cong \pi_{**}(\X \wedge \hz \wedge \hz)\\
		& \cong \pi_{**}(\bigvee_{\alpha \in A}\Sigma^{p_{\alpha},q_{\alpha}}(\X \wedge \hz))\\
		& \cong \bigoplus_{\alpha \in A}\Sigma^{p_{\alpha},q_{\alpha}}\pi_{**}(\X \wedge \hz)\\
		& \cong \bigoplus_{\alpha \in A} \Sigma^{p_{\alpha},q_{\alpha}} H_{**}(k/k).
	\end{align*}
	Now, let $Y$ be any object in $\X-\mo$. Then, 
	\begin{align*}
		H^{iso}_{**}(\X \wedge \hz \wedge Y) &\cong \pi_{**}(\X \wedge \hz \wedge \hz \wedge Y)\\
		& \cong \pi_{**}(\bigvee_{\alpha \in A}\Sigma^{p_{\alpha},q_{\alpha}}(\X \wedge \hz \wedge Y))\\
		& \cong \bigoplus_{\alpha \in A}\Sigma^{p_{\alpha},q_{\alpha}}\pi_{**}(\X \wedge \hz \wedge Y) \\
		& \cong \bigoplus_{\alpha \in A} \Sigma^{p_{\alpha},q_{\alpha}} H^{iso}_{**}(Y) \\
		&\cong \A_{**}(k/k) \otimes_{H_{**}(k/k)} H^{iso}_{**}(Y)
	\end{align*}
	which completes the proof.
\end{proof}

\begin{rem}
	\normalfont
	Note that, by the previous lemma, the map $Y \rightarrow \X \wedge \hz \wedge Y$ induces in isotropic motivic homology a coaction $H^{iso}_{**}(Y) \rightarrow \A_{**}(k/k) \otimes_{H_{**}(k/k)} H^{iso}_{**}(Y)$ which structures $H^{iso}_{**}(Y)$ as a left $\A_{**}(k/k)$-comodule.
\end{rem}

In the next result, we show that, if the homology of an isotropic cellular spectrum $Y$ is free over $H_{**}(k/k)$, then the motivic spectrum $\X \wedge \hz \wedge Y$ is a split $\X \wedge \hz$-module.

\begin{lem}\label{gem}
	Let $ k$ be a flexible field and $Y$ an object in $\X-\mo_{cell}$ such that $H^{iso}_{**}(Y)$ is a free left $H_{**}(k/k)$-module generated by a set of elements $\{x_{\alpha}\}_{\alpha \in A}$, where $x_{\alpha}$ has bidegree $(q_{\alpha})[p_{\alpha}]$. Then, there exists an isomorphism of spectra
	$$\bigvee_{\alpha \in A} \Sigma^{p_{\alpha},q_{\alpha}} (\X \wedge \hz) \xrightarrow{\cong} \X \wedge \hz \wedge Y.$$
\end{lem}
\begin{proof}
	Since $H^{iso}_{**}(Y) \cong \pi_{**}(\X \wedge \hz \wedge Y)$, we can represent each generator $x_{\alpha}$ as a map $\Sigma^{p_{\alpha},q_{\alpha}} \s \rightarrow \X \wedge \hz \wedge Y$, where $(q_{\alpha})[p_{\alpha}]$ is the bidegree of $x_{\alpha}$. For all $\alpha \in A$, this map corresponds bijectively to a map $\Sigma^{p_{\alpha},q_{\alpha}} (\X \wedge \hz) \rightarrow \X \wedge \hz \wedge Y$ of $\X \wedge \hz$-cellular modules. Hence, we get a map
	$$\bigvee_{\alpha \in A} \Sigma^{p_{\alpha},q_{\alpha}} (\X \wedge \hz) \rightarrow \X \wedge \hz \wedge Y$$
	of $\X \wedge \hz$-cellular modules. In order to check that it is an isomorphism, by Proposition \ref{check} it is enough to look at the induced morphisms on homotopy groups. Indeed, we have that on the one hand
	$$\pi_{**}(\bigvee_{\alpha \in A} \Sigma^{p_{\alpha},q_{\alpha}} (\X \wedge \hz)) \cong \bigoplus_{\alpha \in A}\Sigma^{p_{\alpha},q_{\alpha}}  \pi_{**}(\X \wedge \hz) \cong \bigoplus_{\alpha \in A} \Sigma^{p_{\alpha},q_{\alpha}}H_{**}(k/k)$$
	and on the other
	$$\pi_{**}(\X \wedge \hz \wedge Y) \cong \bigoplus_{\alpha \in A} H_{**}(k/k) \cdot x_{\alpha}$$
	by hypothesis. By construction, the map we are considering induces in homotopy groups the homomorphism of $H_{**}(k/k)$-modules 
	$$\pi_{**}(\bigvee_{\alpha \in A} \Sigma^{p_{\alpha},q_{\alpha}} (\X \wedge \hz)) \rightarrow \pi_{**}(\X \wedge \hz \wedge Y)$$
	which sends $1 \in \Sigma^{p_{\alpha},q_{\alpha}} H_{**}(k/k)$ to $x_{\alpha}$ for any $\alpha \in A$, so it is an isomorphism, as we wanted to show.	
\end{proof}

The next lemma provides us with a condition under which the isotropic cohomology of a spectrum is dual to its isotropic homology.

\begin{lem}\label{dual}
	Let $ k$ be a flexible field and $Y$ an object in $\X-\mo$ such that there is an isomorphism $\X \wedge \hz \wedge Y \cong \bigvee_{\alpha \in A} \Sigma^{p_{\alpha},q_{\alpha}} (\X \wedge \hz)$ for some set $A$. Then, for any bidegree $(q)[p]$ there is an isomorphism
	$$H_{iso}^{p,q}(Y) \cong \Hom ^{-p,-q}_{H_{**}(k/k)}(H^{iso}_{**}(Y),H_{**}(k/k)).$$
\end{lem}
\begin{proof}
	Since $\X \wedge \hz \wedge Y \cong \bigvee_{\alpha \in A} \Sigma^{p_{\alpha},q_{\alpha}} (\X \wedge \hz)$ by hypothesis, we have that
	$$H^{iso}_{**}(Y)=[\s^{**},\X \wedge \hz \wedge Y] \cong [\s^{**},\bigvee_{\alpha \in A} \Sigma^{p_{\alpha},q_{\alpha}} (\X \wedge \hz)] \cong \bigoplus_{\alpha \in A} \Sigma^{p_{\alpha},q_{\alpha}} H_{**}(k/k)$$
	from which it follows that
	$$\Hom ^{-p,-q}_{H_{**}(k/k)}(H^{iso}_{**}(Y),H_{**}(k/k)) \cong \prod_{\alpha \in A}  H_{p_{\alpha}-p,q_{\alpha}-q}(k/k).$$
	On the other hand, we have the following chain of isomorphisms:
	\begin{align*}
		H_{iso}^{p,q}(Y) &= [Y,\Sigma^{p,q}(\X \wedge \hz)]\\
		& \cong [\X \wedge \hz \wedge Y,\Sigma^{p,q}(\X \wedge \hz)]_{\X \wedge \hz}\\
		& \cong [\bigvee_{\alpha \in A} \Sigma^{p_{\alpha},q_{\alpha}} (\X \wedge \hz),\Sigma^{p,q}(\X \wedge \hz)]_{\X \wedge \hz}\\
		& \cong [\bigvee_{\alpha \in A}\s^{p_{\alpha},q_{\alpha}},\Sigma^{p,q}(\X \wedge \hz)] \\
		& \cong \prod_{\alpha \in A}  H_{p_{\alpha}-p,q_{\alpha}-q}(k/k)
	\end{align*}
	that concludes the proof. 	
\end{proof}

Now, we need to define a certain concept of finiteness which suits the isotropic environment. 

\begin{dfn}
	\normalfont
We say that a set of bidegrees $\{(q_{\alpha})[p_{\alpha}]\}_{\alpha \in A}$ is isotropically finite type if for any bidegree $(q)[p]$ there are only finitely many $\alpha \in A$ such that $p-p_{\alpha} \geq 2(q-q_{\alpha}) \geq 0$. Moreover, we say that a set of bigraded elements $\{x_{\alpha}\}_{\alpha \in A}$ is isotropically finite type if the corresponding set of bidegrees is so.
\end{dfn}

\begin{lem}\label{ift}
Let $k$ be a flexible field and $\{(q_{\alpha})[p_{\alpha}]\}_{\alpha \in A}$ an isotropically finite type set of bidegrees. Then, for any bidegree $(q)[p]$, the obvious map
$$\pi_{p,q}(\X \wedge \bigvee_{\alpha \in A} \Sigma^{p_{\alpha},q_{\alpha}} \hz) \rightarrow \Hom ^{p,q}_{\A^{**}(k/k)}(H_{iso}^{**}(\bigvee_{\alpha \in A} \Sigma^{p_{\alpha},q_{\alpha}} \hz),H^{**}(k/k))$$
is an isomorphism.
\end{lem}
\begin{proof}
First, note that, for any bidegree $(q)[p]$, one has the following commutative diagram
$$
\xymatrix{
	\pi_{p,q}(\X \wedge \bigvee_{\alpha \in A} \Sigma^{p_{\alpha},q_{\alpha}} \hz) \ar@{->}[r] \ar@{->}[d]&  \Hom ^{p,q}_{\A^{**}(k/k)}(H_{iso}^{**}(\bigvee_{\alpha \in A} \Sigma^{p_{\alpha},q_{\alpha}} \hz),H^{**}(k/k)) \ar@{->}[d]  \\
	\Hom ^{p,q}_{\F}(\bigoplus_{\alpha \in A}  \Sigma^{-p_{\alpha},-q_{\alpha}} \F,H^{**}(k/k)) \ar@{->}[r] & \Hom ^{p,q}_{\A^{**}(k/k)}(\bigoplus_{\alpha \in A}  \Sigma^{-p_{\alpha},-q_{\alpha}} \A^{**}(k/k),H^{**}(k/k)).
}
$$
The left vertical arrow is the isomorphism described by the following chain of equivalences
\begin{align*}
\pi_{p,q}(\X \wedge \bigvee_{\alpha \in A} \Sigma^{p_{\alpha},q_{\alpha}} \hz) &\cong \bigoplus_{\alpha \in A} \pi_{p,q}(\X \wedge \Sigma^{p_{\alpha},q_{\alpha}} \hz) \\
&\cong  \bigoplus_{\alpha \in A} H^{p_{\alpha}-p,q_{\alpha}-q}(k/k)\\
& \cong \prod_{\alpha \in A}H^{p_{\alpha}-p,q_{\alpha}-q} (k/k)\\
& \cong  \prod_{\alpha \in A} \Hom ^{p,q}_{\F}(\Sigma^{-p_{\alpha},-q_{\alpha}} \F,H^{**}(k/k)) \\
	&\cong \Hom ^{p,q}_{\F}(\bigoplus_{\alpha \in A} \Sigma^{-p_{\alpha},-q_{\alpha}} \F,H^{**}(k/k))
\end{align*}
where the identification
$$\bigoplus_{\alpha \in A} H^{p_{\alpha}-p,q_{\alpha}-q}(k/k) \cong \prod_{\alpha \in A}H^{p_{\alpha}-p,q_{\alpha}-q} (k/k)$$
is due to the fact that the set $\{(q_{\alpha})[p_{\alpha}]\}_{\alpha \in A}$ is isotropically finite type, so for any bidegree $(q)[p]$ only for a finite number of $\alpha \in A$ the group $H^{p_{\alpha}-p,q_{\alpha}-q}(k/k)$ is non-zero by Theorem \ref{vis}. The bottom horizontal map is obviously an isomorphism since $\A^{**}(k/k)$ is an $\F$-vector space. The right vertical map is an isomorphism since
\begin{align*}
	\Hom ^{p,q}_{\A^{**}(k/k)}(H^{**}_{iso}(\bigvee_{\alpha} \Sigma^{p_{\alpha},q_{\alpha}} \hz),H^{**}(k/k)) 
	&\cong \Hom ^{p,q}_{\A^{**}(k/k)}(\prod_{\alpha}H^{**}_{iso}(\Sigma^{p_{\alpha},q_{\alpha}} \hz),H^{**}(k/k)) \\
	&= \Hom ^{p,q}_{\A^{**}(k/k)}(\prod_{\alpha} \Sigma^{-p_{\alpha},-q_{\alpha}} \A^{**}(k/k),H^{**}(k/k)) \\
	&\cong \Hom ^{p,q}_{\A^{**}(k/k)}(\bigoplus_{\alpha} \Sigma^{-p_{\alpha},-q_{\alpha}} \A^{**}(k/k),H^{**}(k/k)) 
\end{align*}
where the last isomorphism comes from the fact that the set of bidegrees $\{(q_{\alpha})[p_{\alpha}]\}_{\alpha \in A}$ is isotropically finite type, so for any bidegree $(q)[p]$ only for finitely many $\alpha \in A$ the group
$$\Hom ^{p,q}_{\A^{**}(k/k)}( \Sigma^{-p_{\alpha},-q_{\alpha}}\A^{**}(k/k),H^{**}(k/k))\cong H^{p_{\alpha}-p,q_{\alpha}-q}(k/k)$$ 
is non-trivial by Theorem \ref{vis}. This completes the proof.
\end{proof}

At this point, we are ready to present the structure of the $E_2$-page of the isotropic Adams spectral sequence, which behaves as in the classical case.

\begin{thm}\label{iass}
Let $k$ be a flexible field and $Y$ an object in $\X-\mo_{cell}$ such that $H^{iso}_{**}(Y)$ is a free left $H_{**}(k/k)$-module generated by an isotropically finite type set of elements $\{x_{\alpha}\}_{\alpha \in A}$. Then, the $E_2$-page of the isotropic motivic Adams spectral sequence is described by
$$E_2^{s,t,u} \cong \Ext^{s,t,u}_{\A^{**}(k/k)}(H_{iso}^{**}(Y),H^{**}(k/k)).$$
\end{thm}
\begin{proof}
First, we want to prove by induction that $H_{**}^{iso}((\overline{\X \wedge \hz})^{\wedge s} \wedge Y)$ is a free left $H_{**}(k/k)$-module generated by an isotropically finite type set of elements $\{x_{\alpha}\}_{\alpha \in A_s}$ for any $s \geq 0$. The induction basis is guaranteed by hypothesis after setting $A_0=A$. Suppose the statement is true at the $s-1$ stage, i.e. $H_{**}^{iso}((\overline{\X \wedge \hz})^{\wedge s-1} \wedge Y) \cong \bigoplus_{\alpha \in A_{s-1}} \Sigma^{1-s,0}H_{**}(k/k) \cdot x_{\alpha}$. Then, by Lemma \ref{kun}, the map $(\overline{\X \wedge \hz})^{\wedge s-1} \wedge Y \rightarrow \X \wedge \hz \wedge(\overline{\X \wedge \hz})^{\wedge s-1}   \wedge Y$ induces in isotropic motivic homology the monomorphism
$$\bigoplus_{\alpha \in A_{s-1}} \Sigma^{1-s,0}H_{**}(k/k) \cdot x_{\alpha} \rightarrow  \bigoplus_{\alpha \in A_{s-1}} \Sigma^{1-s,0}\A_{**}(k/k) \cdot x_{\alpha}.$$
Hence, the standard Adams resolution induces for any $p$ and $q$ a short exact sequence
$$0 \rightarrow H_{p,q}^{iso}((\overline{\X \wedge \hz})^{\wedge s-1} \wedge Y)  \rightarrow H_{p,q}^{iso}(\X \wedge \hz \wedge(\overline{\X \wedge \hz})^{\wedge s-1} \wedge Y) \rightarrow  H_{p-1,q}^{iso}((\overline{\X \wedge \hz})^{\wedge s} \wedge Y) \rightarrow 0.$$
Now, note that, by the very structure of the dual of the isotropic motivic Steenrod algebra, $\A_{**}(k/k)$ is freely generated over $H_{**}(k/k)$ by a set of generators $\{1,y_{\beta}\}_{\beta \in B}$ which is finite in each bidegree and such that $p_{\beta} \geq 2q_{\beta} \geq 0$ for any $\beta \in B$, where $(q_{\beta})[p_{\beta}]$ is the bidegree of $y_{\beta}$. Hence, the set $\{y_{\beta}x_{\alpha}\}_{\beta \in B,\alpha \in A_{s-1}}$ is isotropically finite type and freely generates $H_{**}^{iso}((\overline{\X \wedge \hz})^{\wedge s} \wedge Y)$ over $H_{**}(k/k)$, i.e. 
$$H_{**}^{iso}((\overline{\X \wedge \hz})^{\wedge s} \wedge Y) \cong \bigoplus_{\beta \in B,\alpha \in A_{s-1}} \Sigma^{-s,0}H_{**}(k/k) \cdot y_{\beta}x_{\alpha}.$$
Therefore, Lemma \ref{gem} implies that all $\X \wedge \hz \wedge(\overline{\X \wedge \hz})^{\wedge s} \wedge Y$ are wedges of appropriately shifted $\X \wedge \hz$. More precisely, for any $s \geq 0$, there exists an isomorphism
$$\X \wedge \bigvee_{\alpha \in A_s} \Sigma^{p_{\alpha}-s,q_{\alpha}} \hz \xrightarrow{\cong} \X \wedge \hz \wedge(\overline{\X \wedge \hz})^{\wedge s} \wedge Y$$ 
where $A_s=B \times A_{s-1}$, from which we deduce, using Lemma \ref{ift}, that the $E_1$-page of the isotropic Adams spectral sequence can be described by
$$E_1^{s,t,u} \cong \pi_{t-s,u}(\X \wedge \hz\wedge (\overline{\X \wedge \hz})^{\wedge s} \wedge Y) \cong \Hom ^{t,u}_{\A^{**}(k/k)}(\bigoplus_{\alpha \in A_s} \Sigma^{-p_{\alpha},-q_{\alpha}} \A^{**}(k/k),H^{**}(k/k)).$$
Moreover, note that
$$0 \leftarrow H_{iso}^{**}(Y) \leftarrow \bigoplus_{\alpha \in A_0} \Sigma^{-p_{\alpha},-q_{\alpha}} \A^{**}(k/k) \leftarrow \bigoplus_{\alpha \in A_1} \Sigma^{-p_{\alpha},-q_{\alpha}} \A^{**}(k/k) \leftarrow \dots$$
is a free $\A^{**}(k/k)$-resolution of $H_{iso}^{**}(Y)$. Thus, for any $s,t,u$ we have an isomorphism
$$E_2^{s,t,u} \cong \Ext^{s,t,u}_{\A^{**}(k/k)}(H_{iso}^{**}(Y),H^{**}(k/k))$$
as we aimed to prove. 
\end{proof}

By using the isotropic motivic Adams spectral sequence, in \cite{T} we computed the isotropic motivic homotopy groups of the sphere spectrum which can be identified with the $E_2$-page of the classical Adams spectral sequence.

\begin{thm}
	Let $k$ be a flexible field. Then, the stable motivic homotopy groups of the $\hz$-completed isotropic sphere spectrum are completely described by
	$$\pi_{*,*'}(\X^{\wedge}_{\hz}) \cong \Ext_{\G^{**}}^{2*'-*,2*',*'}(\F,\F) \cong \Ext_{\A^*}^{2*'-*,*'}(\F,\F).$$
\end{thm}
\begin{proof}
	See \cite[Theorem 5.7]{T}. 
\end{proof}

\section{The motivic Brown-Peterson spectrum}

In this section, we recall from \cite{Ve} the construction of the motivic Brown-Peterson spectrum. Moreover, we compute its isotropic homology and homotopy, which will be useful later on for the construction of the isotropic motivic Adams-Novikov spectral sequence and so for the proofs of our main results.

\begin{dfn}
	\normalfont
Let $\mgl_{(2)}$ be the motivic algebraic cobordism spectrum (see \cite[Section 6.3]{V0}) localised at 2. Then, following \cite[Section 5]{Ve} one defines the motivic Brown-Peterson spectrum at the prime 2 as the colimit of the diagram in $\SH(k)$
$$\dots \rightarrow \mgl_{(2)} \xrightarrow{e_{(2)}} \mgl_{(2)} \xrightarrow{e_{(2)}} \mgl_{(2)} \rightarrow \dots$$
where $e_{(2)}$ is the motivic Quillen idempotent. 
\end{dfn}

Note, in particular, that $\mbp$ is a homotopy commutative ring spectrum and a direct summand of $\mgl_{(2)}$.

\begin{prop}
Let $k$ be a flexible field. Then, there is an isomorphism of $H^{**}(k/k)$-modules
$$H_{iso}^{**}(\mgl) \cong H_{iso}^{**}(\bgl) \cong H^{**}(k/k)[c_1,c_2,\dots]$$
and an isomorphism of $H_{**}(k/k)$-algebras
$$H^{iso}_{**}(\mgl) \cong H^{iso}_{**}(\bgl) \cong H_{**}(k/k)[b_1,b_2,\dots]$$
where $c_i$ is the $i$th Chern class in $H_{iso}^{2i,i}(\bgl)$ and $b_i \in H^{iso}_{2i,i}(\bgl)$ is the dual of $c_1^i$ with respect to the monomial basis, for any $i$. 
\end{prop}
\begin{proof}
	First, note that the maps $P^1 \rightarrow P^{\infty}$ and $\hz \rightarrow \X \wedge \hz$ induce a commutative square
	$$
	\xymatrix{
		H^{**}(P^{\infty}) \ar@{->}[r] \ar@{->}[d]&   H^{**}_{iso}(P^{\infty})  \ar@{->}[d]\\
		H^{**}(P^1) \ar@{->}[r] & H^{**}_{iso}(P^1)  
	}
	$$
	where the left vertical morphism is the projection $H^{**}(k)[c] \rightarrow H^{**}(k)[c]/(c^2)$ and $c$ is the only non zero class in $H^{2,1}(P^{\infty}) \cong H^{2,1}(P^1) \cong \Z/2$. If we also denote by $c$ the images of $c$ under the horizontal maps in isotropic motivic cohomology, then the right vertical homomorphism is given by the projection $$H^{**}(k/k)[c] \rightarrow H^{**}(k/k)[c]/(c^2).$$
	Hence, $\X \wedge \hz$ is an oriented motivic spectrum (see \cite[Definition 3.1]{Ve}) and the statement follows immediately from \cite[Proposition 6.2]{NSO}.
\end{proof}
	
Following \cite[Section 6]{H}, let $h:L \rightarrow \F[b_1,b_2,\dots]$ be the homomorphism from the Lazard ring $L$ classifying the formal group law on $\F[b_1,b_2,\dots]$ which is isomorphic to the additive one via the exponential $\sum_{n \geq 0}b_nx^{n+1}$. Lazard's theorem implies that $h(L)$ is a polynomial subring $\F[b'_n| n \neq 2^r-1]$, where $b'_n \equiv b_n$ modulo decomposables. Denote by $\pi:\F[b_1,b_2,\dots] \rightarrow h(L)$ a retraction of the inclusion.

In the next proposition, we give a description of isotropic homology and cohomology of the algebraic cobordism spectrum $\mgl$.

\begin{prop}
	Let $k$ be a flexible field. Then, the coaction $\Delta: H^{iso}_{**}(\mgl)  \rightarrow \A_{**}(k/k) \otimes_{H_{**}(k/k)}H^{iso}_{**}(\mgl) $ factors through $H_{**}(k/k) \otimes_{\F} \G_{**} \otimes_{\F} \F[b_1,b_2,\dots]$ and the composition
	$$H^{iso}_{**}(\mgl) \xrightarrow{\Delta} H_{**}(k/k) \otimes_{\F} \G_{**} \otimes_{\F} \F[b_1,b_2,\dots] \xrightarrow{id \otimes \pi} H_{**}(k/k) \otimes_{\F} \G_{**} \otimes_{\F} h(L)$$
	is an isomorphism of left $\A_{**}(k/k)$-comodule algebras.
	Dually, the map
	$$H^{**}(k/k) \otimes_{\F} \G^{**} \otimes_{\F} h(L)^{\lor} \rightarrow H^{**}_{iso}(\mgl)$$
	is an isomorphism of left $\A^{**}(k/k)$-module coalgebras.
\end{prop}
\begin{proof}
	Since $\hz \wedge \mgl$ is a split $\hz$-module (see the remark after \cite[Definition 5.4]{H}), from \cite[Lemma 5.2]{H} we deduce that 
	$$H^{iso}_{**}(\mgl)\cong \pi_{**}(\X \wedge \hz) \otimes_{\pi_{**}(\hz)} \pi_{**}(\hz \wedge \mgl)  \cong H_{**}(k/k) \otimes_{H_{**}(k)} H_{**}(\mgl)$$
	as an $H_{**}(k/k)$-algebra. From \cite[Theorem 6.5]{H} we know that the coaction $\Delta: H_{**}(\mgl)  \rightarrow \A_{**}(k) \otimes_{H_{**}(k)}H_{**}(\mgl) $ factors through $\Pa_{**}\otimes_{\F} \F[b_1,b_2,\dots]$ and the composition
	$$H_{**}(\mgl) \xrightarrow{\Delta} \Pa_{**} \otimes_{\F} \F[b_1,b_2,\dots] \xrightarrow{id \otimes \pi} \Pa_{**}\otimes_{\F} h(L)$$
	is an isomorphism of left $\A_{**}(k)$-comodule algebras, where $\Pa_{**}$ is the subalgebra of $\A_{**}(k)$ defined by $H_{**}(k)[\xi_1,\xi_2,\dots]$. By tensoring the previous composition with $H_{**}(k/k)$ over $H_{**}(k)$ we get the desired isomorphism, which completes the first part. The second part follows easily, since $\G_{**} \otimes_{\F} h(L)$ is isotropically finite type, from Lemmas \ref{gem} and \ref{dual} by dualizing the homology isomorphism. 	
\end{proof}

The next result provides us with the structure of isotropic homology and cohomology of the motivic Brown-Peterson spectrum $\mbp$.

\begin{prop}\label{imbp}
	Let $k$ be a flexible field. Then, the isotropic motivic homology of $\mbp$ is described as a left $\A_{**}(k/k)$-comodule by
	$$H^{iso}_{**}(\mbp) \cong H_{**}(k/k) \otimes_{\F} \G_{**}.$$
	Dually, the isotropic motivic cohomology of $\mbp$ is described as a left $\A^{**}(k/k)$-module by
	$$H^{**}_{iso}(\mbp) \cong H^{**}(k/k) \otimes_{\F} \G^{**}.$$
\end{prop}
\begin{proof}
	From \cite[Remark 6.20]{H}, one knows that $\mbp$ is equivalent to $\mgl_{(2)}/x$ where $x$ is any maximal $h$-regular sequence, i.e. a sequence of homogeneous elements in $L$ such that $h(x)$ is a regular sequence in $h(L)$ which generates the maximal ideal. Therefore, \cite[Theorem 6.11]{H} implies that there exists an isomorphism of $\A_{**}(k)$-comodules
	$$H_{**}(\mbp) \cong \Pa_{**}.$$
	Since $\hz \wedge \mbp$ is a split $\hz$-module, we deduce from \cite[Lemma 5.2]{H} that 
	$$H_{**}^{iso}(\mbp) \cong H_{**}(k/k) \otimes_{H_{**}(k)} H_{**}(\mbp) \cong H_{**}(k/k) \otimes_{H_{**}(k)} \Pa_{**} \cong H_{**}(k/k) \otimes_{\F} \G_{**}$$
	which proves the first part. The second part follows again from dualization, since $\G_{**}$ is isotropically finite type, by Lemmas \ref{gem} and \ref{dual}.
\end{proof}

Later on, we will also need the isotropic homology and cohomology of $\mbp \wedge \mbp$, which is reported in the following proposition.

\begin{prop}\label{mbp2}
Let $k$ be a flexible field. Then, the isotropic motivic homology of $\mbp \wedge \mbp$ is described as a left $\A_{**}(k/k)$-comodule by
$$H^{iso}_{**}(\mbp \wedge \mbp) \cong H_{**}(k/k) \otimes_{\F} \G_{**} \otimes_{\F} \G_{**}.$$
Dually, the isotropic motivic cohomology of $\mbp \wedge \mbp$ is described as a left $\A^{**}(k/k)$-module by
$$H^{**}_{iso}(\mbp \wedge \mbp) \cong H^{**}(k/k) \otimes_{\F} \G^{**}\otimes_{\F} \G^{**}.$$
\end{prop}
\begin{proof}
Since $\hz \wedge \mbp$ is a split $\hz$-module, by \cite[Lemma 5.2]{H} and Proposition \ref{imbp} we obtain that 
$$H^{iso}_{**}(\mbp \wedge \mbp) \cong (H_{**}(k/k) \otimes_{\F} \G_{**}) \otimes_{H_{**}(k/k)} (H_{**}(k/k) \otimes_{\F} \G_{**})  \cong H_{**}(k/k) \otimes_{\F} \G_{**} \otimes_{\F} \G_{**}.$$
The description of the isotropic cohomology follows again by dualizing the homology isomorphism.
\end{proof}

Now, we compute the isotropic stable homotopy groups of $\mbp$ by using the isotropic Adams spectral sequence developed in the previous section.

\begin{thm}\label{hgmbp}
	Let $k$ be a flexible field. Then, the isotropic motivic homotopy groups of $\mbp$ are described by
	$$\pi^{iso}_{**}(\mbp) \cong \F.$$
\end{thm}
\begin{proof}
	First, note that, by Proposition \ref{imbp}, $H^{iso}_{**}(\mbp)$ is freely generated over $H_{**}(k/k)$ by $\G_{**}$ which is isotropically finite type. Hence, Theorem \ref{iass} implies that the $E_2$-page of the isotropic motivic Adams spectral sequence for $\X \wedge \mbp$ is given by
	$$E_2^{s,t,u} \cong \Ext_{\A^{**}(k/k)}^{s,t,u}(H_{iso}^{**}(\mbp),H^{**}(k/k)).$$
	Now, we deduce from Proposition \ref{imbp} and \cite[Theorem 5.4]{T} that
	\begin{align*}
     \Ext_{\A^{**}(k/k)}^{s,t,u}(H_{iso}^{**}(\mbp),H^{**}(k/k)) & \cong \Ext_{\A^{**}(k/k)}^{s,t,u}(H^{**}(k/k) \otimes_{\F} \G^{**},H^{**}(k/k)) \\
     &\cong \Ext_{\G^{**}}^{s,t,u}(\G^{**},\F) \\
     &\cong \Ext_{\F}^{s,t,u}(\F,\F) \cong  
     \begin{cases}
     \F & if \: s=t=u=0\\
     0 & otherwise	
     	\end{cases}.
     \end{align*}
	Therefore, the $E_2$-page of the isotropic Adams spectral sequence for $\X \wedge \mbp$ is concentrated just in the tridegree $(0,0,0)$, from which it follows that all differentials from the second on are trivial. Thus, the Mittag-Leffler condition is clearly satisfied, and so strong convergence holds by Proposition \ref{conv}. Then, it immediately follows from Remark \ref{abc} and the fact that $\mbp$ is $\eta$-complete that
	 $$\pi^{iso}_{**}(\mbp) \cong \pi_{**}(\X \wedge \mbp) \cong \F$$
	 which completes the proof.
\end{proof}

In the following sections, it will be also useful to know the isotropic homotopy groups of $\mbp \wedge \mbp$ that we compute in the next result.

\begin{thm}\label{hgmbp2}
	Let $k$ be a flexible field. Then, the isotropic motivic homotopy groups of $\mbp \wedge \mbp$ are described by
	$$\pi^{iso}_{**}(\mbp \wedge \mbp) \cong \G_{**}.$$
\end{thm}
\begin{proof}
	The proof of this theorem goes along the lines of the previous one. Since $H_{**}^{iso}(\mbp \wedge \mbp) \cong H_{**}(k/k) \otimes_{\F} \G_{**}\otimes_{\F} \G_{**}$ by Proposition \ref{mbp2} and $\G_{**}\otimes_{\F} \G_{**}$ is isotropically finite type, by Theorem \ref{iass} we have that the $E_2$-page of the isotropic Adams spectral sequence for $\X \wedge \mbp \wedge \mbp$ is provided by
	$$E_2^{s,t,u} \cong \Ext_{\A^{**}(k/k)}^{s,t,u}(H_{iso}^{**}(\mbp \wedge \mbp),H^{**}(k/k)).$$
	Again, we note that by \cite[Theorem 5.4]{T}
	\begin{align*}
		\Ext_{\A^{**}(k/k)}^{s,t,u}(H_{iso}^{**}(\mbp \wedge \mbp),H^{**}(k/k)) & \cong \Ext_{\A^{**}(k/k)}^{s,t,u}(H^{**}(k/k) \otimes_{\F} \G^{**} \otimes_{\F} \G^{**},H^{**}(k/k)) \\
		&\cong \Ext_{\G^{**}}^{s,t,u}(\G^{**} \otimes_{\F} \G^{**},\F) \\
		&\cong \Ext_{\F}^{s,t,u}(\G^{**},\F) \cong  
		\begin{cases}
			\G_{t,u} & if \: s=0\\
			0 & if \: s \neq 0	
		\end{cases}.
	\end{align*}
	In particular, since $\G_{**}$ is concentrated on the slope $2$ line, we have that all differentials from the second on are trivial by degree reasons. Hence, Mittag-Leffler condition is met, which implies that the spectral sequence is strongly convergent. From all this, it follows as above that
	$$\pi^{iso}_{**}(\mbp \wedge \mbp) \cong \pi_{**}(\X \wedge \mbp \wedge \mbp) \cong  \G_{**}$$
	that is what we aimed to show.
\end{proof}

\section{The category of isotropic cellular MBP-modules}

In this section we start by providing $\X \wedge \mbp$ with an $E_{\infty}$-ring structure. This allows us to talk about the stable $\infty$-category of $\X \wedge \mbp$-modules, i.e. $\X \wedge \mbp - \mo$, and its cellular part, i.e. $\X \wedge \mbp- \mo_{cell}$. Our aim is to focus on isotropic cellular $\mbp$-modules, which is the same as cellular $\X \wedge \mbp$-modules. In particular, we completely describe the category $\X \wedge \mbp-\mo_{cell}$ in algebraic terms. This section is structured along the lines of \cite[Section 3]{GWX}. Therefore, before each result we indicate the one from \cite{GWX} it corresponds to. We hope this could clearly shed light on the deep parallelism between \cite{GWX} and this work.

\begin{prop}
The homotopy commutative ring structure on $\X \wedge \mbp$ extends to an $E_{\infty}$-ring structure.
\end{prop}
\begin{proof}
It follows from \cite[Proposition 1.4.4.11]{Lu} that there exists a $t$-structure on $\X-\mo$ with non-negative part generated by $\X^{2n,n}$ for any $n \in \Z$. By \cite[Theorem A.1]{BKWX}, $\X \wedge \mgl$ belongs to the non-negative parte of this $t$-structure, and so also $\X \wedge \mbp$ does. On the other hand, one deduces from Theorem \ref{hgmbp} and \cite[Lemma 2.4]{BKWX} that $\X \wedge \mbp$ belongs to the non-positive part too. Hence, $\X \wedge \mbp$ is a homotopy commutative ring spectrum in the heart of the above mentioned $t$-structure, which means that it is an $E_{\infty}$-ring spectrum.\footnote{I am grateful to Tom Bachmann for this argument.} 	
	\end{proof}

Once we know that $\X \wedge \mbp$ is a motivic $E_{\infty}$-ring spectrum, we can consider the stable $\infty$-category of $\X \wedge \mbp$-modules and its homotopy category which is tensor triangulated. In particular, we focus on its cellular part.

\begin{prop}\label{gem2}
	Let $k$ be a flexible field and $Y$ an object in $\X \wedge \mbp-\mo_{cell}$ such that $\pi_{**}(Y)$ is isomorphic to the $\F$-vector space $\bigoplus_{\alpha \in A} \Sigma^{p_{\alpha},q_{\alpha}} \F$. Then, there exists an isomorphism of spectra
	$$\bigvee_{\alpha \in A} \Sigma^{p_{\alpha},q_{\alpha}} (\X \wedge \mbp) \xrightarrow{\cong} Y.$$
\end{prop}
\begin{proof}
	We follow the lines of the proof of Lemma \ref{gem}. Each generator of $\pi_{**}(Y)$ represents a map $\Sigma^{p_{\alpha},q_{\alpha}} \s \rightarrow Y$. For all $\alpha \in A$, this map corresponds bijectively to a map $\Sigma^{p_{\alpha},q_{\alpha}} (\X \wedge \mbp) \rightarrow Y$ of $\X \wedge \mbp$-cellular modules. Hence, we get a map
	$$\bigvee_{\alpha \in A} \Sigma^{p_{\alpha},q_{\alpha}} (\X \wedge \mbp) \rightarrow Y$$
	of $\X \wedge \mbp$-cellular modules that induces an isomorphism on homotopy groups since $\pi_{**}(\X \wedge \mbp) \cong \F$ by Theorem \ref{hgmbp}. Therefore, it follows from Proposition \ref{check} that the above map is an isomorphism of spectra, which completes the proof.
\end{proof}

The previous result implies the following corollary that corresponds to \cite[Corollary 3.3]{GWX}. 

\begin{cor}\label{gem3}
	Let $k$ be a flexible field and $X$ and $Y$ be objects in $\X \wedge \mbp-\mo_{cell}$. Then, 
	$$[X,Y]_{\X \wedge \mbp} \cong \Hom ^{0,0}_{\F}(\pi_{**}(X),\pi_{**}(Y)).$$
\end{cor}
\begin{proof}
It follows from Proposition \ref{gem2} that $X \cong \bigvee_{\alpha \in A} \Sigma^{p_{\alpha},q_{\alpha}} (\X \wedge \mbp)$ and $Y \cong \bigvee_{\beta \in B} \Sigma^{p_{\beta},q_{\beta}} (\X \wedge \mbp)$ for some sets $A$ and $B$. Then, we have that
\begin{align*}
[X,Y]_{\X \wedge \mbp} & \cong [\bigvee_{\alpha \in A} \Sigma^{p_{\alpha},q_{\alpha}} \s,\bigvee_{\beta \in B} \Sigma^{p_{\beta},q_{\beta}} (\X \wedge \mbp)] \\
& \cong \prod_{\alpha \in A} \bigoplus_{\beta \in B} \pi_{p_{\alpha}-p_{\beta},q_{\alpha}-q_{\beta}}(\X \wedge \mbp)\\
& \cong \prod_{\alpha \in A} \bigoplus_{\beta \in B} \Sigma^{p_{\alpha}-p_{\beta},q_{\alpha}-q_{\beta}}\F\\
& \cong \Hom^{0,0}_{\F}(\bigoplus_{\alpha \in A} \Sigma^{p_{\alpha},q_{\alpha}}\F,\bigoplus_{\beta \in B} \Sigma^{p_{\beta},q_{\beta}}\F)\\
& \cong \Hom ^{0,0}_{\F}(\pi_{**}(X),\pi_{**}(Y))
\end{align*}
which concludes the proof.
\end{proof}

The next theorem, which corresponds to \cite[Theorem 3.8]{GWX}, identifies $\X \wedge \mbp-\mo_{cell}$ with the category of bigraded $\F$-vector spaces that we denote by $\F-\mo_{**}$.

\begin{thm}\label{mbpcell}
	Let $k$ be a flexible field. Then, the functor 
	$$\pi_{**}:\X \wedge \mbp-\mo_{cell} \xrightarrow{\cong} \F-\mo_{**}$$
	is an equivalence of categories.
\end{thm}
\begin{proof}
It follows immediately from Proposition \ref{gem2} and Corollary \ref{gem3}. 
\end{proof}

\begin{rem}
	\normalfont
	We want to highlight that the equivalence provided by the previous theorem is actually an equivalence of triangulated categories, where $\F-\mo_{**}$ is structured as a triangulated category in the obvious way. More precisely, the translation functor is the suspension $\Sigma^{1,0}$ and distinguished triangles are of the form
	$$V \xrightarrow{f} W \rightarrow coker(f) \oplus \Sigma^{1,0}ker(f) \rightarrow \Sigma^{1,0}V$$ 
	where $f$ is a morphism of bigraded $\F$-vector spaces.
\end{rem}

\section{The category of isotropic cellular spectra}

This section is devoted to the understanding of the structure of the category $\X-\mo_{cell}$ that is, as we have already noticed, the category of cellular isotropic spectra $\SH(k/k)_{cell}$. We give a nice algebraic description of this category based on the dual of the topological Steenrod algebra. The results here are the isotropic versions of the ones in \cite[Sections 4 and 5]{GWX}. Therefore, the proofs we provide are isotropic adaptations of the respective ones in \cite{GWX}.

In the next lemma, which corresponds to \cite[Lemma 5.1]{GWX}, we compute the $\mbp$-homology of isotropic $\mbp$-cellular spectra.

\begin{lem}\label{inj}
	Let $k$ be a flexible field. Then, for any $I \in \X \wedge \mbp-\mo_{cell}$ there is an isomorphism of left $\G_{**}$-comodules
	$$\mbp_{**}(I) \cong \G_{**} \otimes_{\F} \pi_{**}(I).$$
\end{lem}
\begin{proof}
	Since the motivic spectrum $I$ is by hypothesis in $\X \wedge \mbp-\mo_{cell}$, we deduce from Theorem \ref{mbpcell} that $I \cong \bigvee_{\alpha \in A} \Sigma^{p_{\alpha},q_{\alpha}}(\X \wedge \mbp)$ for some set $A$. Therefore, by Theorem \ref{hgmbp2} one has that
	\begin{align*}
	\mbp_{**}(I)=\pi_{**}(\mbp \wedge I) &\cong \pi_{**}(\bigvee_{\alpha \in A} \Sigma^{p_{\alpha},q_{\alpha}}(\X \wedge \mbp \wedge \mbp)) \\
	&\cong \bigoplus_{\alpha \in A} \Sigma^{p_{\alpha},q_{\alpha}}\pi_{**}(\X \wedge \mbp \wedge \mbp) \\
	&\cong \bigoplus_{\alpha \in A} \Sigma^{p_{\alpha},q_{\alpha}}\G_{**} \cong \G_{**}\otimes_{\F} V 
	\end{align*}
where $V\cong \bigoplus_{\alpha \in A} \Sigma^{p_{\alpha},q_{\alpha}} \F$. Now, note that by Theorem \ref{hgmbp}
$$\pi_{**}(I) \cong \bigoplus_{\alpha \in A}\Sigma^{p_{\alpha},q_{\alpha}}\pi_{**}(\X \wedge \mbp) \cong V.$$
It follows that	
$$\mbp_{**}(I) \cong \G_{**} \otimes_{\F} \pi_{**}(I)$$
that is what we aimed to show. 	
\end{proof}

The following lemma, which corresponds to \cite[Lemma 5.3]{GWX}, describes algebraically the hom-sets from isotropic cellular spectra to isotropic $\mbp$-cellular spectra.

\begin{lem}\label{injiso}
	Let $k$ be a flexible field. Then, for any $X \in \X-\mo_{cell}$ and $I \in \X \wedge \mbp-\mo_{cell}$ there is an isomorphism
	$$[X,I] \cong \Hom ^{0,0}_{\G_{**}}(\mbp_{**}(X),\mbp_{**}(I)).$$
\end{lem}
\begin{proof}
	By Theorem \ref{mbpcell} and Lemma \ref{inj}, we have the following sequence of isomorphisms
\begin{align*}
	[X,I] &\cong [\X \wedge \mbp \wedge X,I]_{\X \wedge \mbp} \\
	&\cong \Hom ^{0,0}_{\F}(\pi_{**}(\X \wedge \mbp \wedge X),\pi_{**}(I))\\
	&\cong \Hom ^{0,0}_{\G_{**}}(\pi_{**}(\X \wedge \mbp \wedge X),\G_{**} \otimes_{\F} \pi_{**}(I)) \\
	& \cong \Hom ^{0,0}_{\G_{**}}(\mbp_{**}(X),\mbp_{**}(I))
	\end{align*}
	which concludes the proof.	
\end{proof}

Before constructing the isotropic version of the Adams-Novikov spectral sequence we need the following lemma.

\begin{lem}\label{cas}
	Let $k$ be a flexible field and $Y$ an object in $\X-\mo$. Then, for any $s \geq 0$, there exist isomorphisms
	$$\mbp_{**}((\overline{\X \wedge \mbp})^{\wedge s} \wedge Y) \cong \Sigma^{-s,0} \overline{\G_{**}}^{\otimes s} \otimes_{\F} \mbp_{**}(Y)$$
	and
$$\mbp_{**}(\X \wedge \mbp\wedge (\overline{\X \wedge \mbp})^{\wedge s} \wedge Y) \cong \Sigma^{-s,0} \G_{**} \otimes_{\F} \overline{\G_{**}}^{\otimes s} \otimes_{\F} \mbp_{**}(Y).$$
\end{lem}
\begin{proof}
First, note that by arguments similar to the ones in Lemma \ref{kun} we have an isomorphism
$$\mbp_{**}(\X \wedge \mbp\wedge (\overline{\X \wedge \mbp})^{\wedge s} \wedge Y) \cong \G_{**} \otimes_{\F} \mbp_{**}((\overline{\X \wedge \mbp})^{\wedge s} \wedge Y) $$
for any isotropic spectrum $Y$ and any $s \geq 0$. So, we only need to prove the first part of the statement. We achieve this by an induction argument, after noting that obviously the statement holds for $s=0$.

Now, suppose the statement holds for $s-1$, i.e. 
$$\mbp_{**}( (\overline{\X \wedge \mbp})^{\wedge s-1} \wedge Y) \cong \Sigma^{1-s,0}  \overline{\G_{**}}^{\otimes s-1} \otimes_{\F} \mbp_{**}(Y)$$
and 
$$\mbp_{**}(\X \wedge \mbp\wedge (\overline{\X \wedge \mbp})^{\wedge s-1} \wedge Y) \cong \Sigma^{1-s,0} \G_{**} \otimes_{\F} \overline{\G_{**}}^{\otimes s-1} \otimes_{\F} \mbp_{**}(Y).$$
Then, the distinguished triangle in $\SH(k)$
$$(\overline{\X \wedge \mbp})^{\wedge s} \wedge Y \rightarrow (\overline{\X \wedge \mbp})^{\wedge s-1}\wedge Y \rightarrow \X \wedge \mbp\wedge (\overline{\X \wedge \mbp})^{\wedge s-1} \wedge Y\rightarrow \Sigma^{1,0}(\overline{\X \wedge \mbp})^{\wedge s}\wedge Y$$	
induces in $\mbp$-homology the short exact sequence
$$0 \rightarrow  \Sigma^{1-s,0}  \overline{\G_{**}}^{\otimes s-1} \otimes_{\F} \mbp_{**}(Y) \rightarrow \Sigma^{1-s,0} \G_{**} \otimes_{\F} \overline{\G_{**}}^{\otimes s-1} \otimes_{\F} \mbp_{**}(Y)$$
$$ \rightarrow \Sigma^{1,0} \mbp_{**}((\overline{\X \wedge \mbp})^{\wedge s} \wedge Y) \rightarrow 0.$$
It follows that 
$$\mbp_{**}((\overline{\X \wedge \mbp})^{\wedge s} \wedge Y) \cong \Sigma^{-s,0} \overline{\G_{**}}^{\otimes s} \otimes_{\F} \mbp_{**}(Y)$$
and
$$\mbp_{**}(\X \wedge \mbp\wedge (\overline{\X \wedge \mbp})^{\wedge s} \wedge Y) \cong \Sigma^{-s,0} \G_{**} \otimes_{\F} \overline{\G_{**}}^{\otimes s} \otimes_{\F} \mbp_{**}(Y)$$
which completes the proof. 
\end{proof}

We are now ready to construct the isotropic Adams-Novikov spectral sequence, which corresponds to \cite[Theorem 5.6]{GWX}. Before proceeding, we would like to fix some notations. 

\begin{dfn}\label{dc}
	\normalfont
Let $X$ be an isotropic spectrum. The Chow-Novikov degree of $\mbp_{p,q}(X)$ is the integer $p-2q$. We denote by $\X-\mo^b_{cell}$ the category of bounded isotropic cellular spectra, i.e. isotropic cellular spectra whose $\mbp$-homology is non-trivial only for a finite number of Chow-Novikov degrees.
\end{dfn}

\begin{thm}
		Let $k$ be a flexible field and $X$ and $Y$ objects in $\X-\mo^b_{cell}$. Then, there is a strongly convergent spectral sequence
		$$E_2^{s,t,u} \cong \Ext^{s,t,u}_{\G_{**}}(\mbp_{**}(X),\mbp_{**}(Y)) \Longrightarrow [\Sigma^{t-s,u}X,Y^{\wedge}_{\hz}].$$
\end{thm}
\begin{proof}
Consider the Postnikov system in $\X-\mo_{cell}$	
	$$
	\xymatrix{
		\dots \ar@{->}[r] &  (\overline{\X \wedge \mbp})^{\wedge s} \wedge Y \ar@{->}[r] \ar@{->}[d] &  \dots \ar@{->}[r]   & \overline{\X \wedge \mbp} \wedge Y \ar@{->}[r] \ar@{->}[d]	 & Y \ar@{->}[d]  \\
		 &	\X \wedge \mbp\wedge (\overline{\X \wedge \mbp})^{\wedge s} \wedge Y \ar@{->}[ul]^{[1]} & & \X \wedge \mbp \wedge \overline{\X \wedge \mbp} \wedge Y \ar@{->}[ul]^{[1]}  &	\X \wedge \mbp \wedge Y \ar@{->}[ul]^{[1]} 
	}
	$$
where $\overline{\X \wedge \mbp}$ is defined by the following distinguished triangle in $\SH(k)$
$$\overline{\X \wedge \mbp} \rightarrow \s \rightarrow \X \wedge \mbp \rightarrow \Sigma^{1,0} \overline{\X \wedge \mbp}.$$
If we apply the functor $[\Sigma^{**}X,-]$ we get an unrolled exact couple
$$
\xymatrix{
	  \dots \ar@{->}[r]   & [\Sigma^{**}X,\overline{\X \wedge \mbp} \wedge Y] \ar@{->}[r] \ar@{->}[d]	 & [\Sigma^{**}X,Y] \ar@{->}[d]  \\
  & [\Sigma^{**}X,\X \wedge \mbp \wedge \overline{\X \wedge \mbp} \wedge Y] \ar@{->}[ul]^{[1]}  &	[\Sigma^{**}X,\X \wedge \mbp \wedge Y] \ar@{->}[ul]^{[1]} 
}
$$	
that induces a spectral sequence with $E_1$-page given by
$$E_1^{s,t,u} \cong [\Sigma^{t-s,u}X,\X \wedge \mbp\wedge (\overline{\X \wedge \mbp})^{\wedge s} \wedge Y]$$
and first differential
$$d_1^{s,t,u}:E_1^{s,t,u} \rightarrow E_1^{s+1,t,u}.$$
This is what we call the isotropic Adams-Novikov spectral sequence. Note that by Lemmas \ref{injiso} and \ref{cas} the $E_1$-page has a nice description provided by
$$E_1^{s,t,u} \cong \Hom ^{t,u}_{\G_{**}}(\mbp_{**}(X),\G_{**} \otimes_{\F} \overline{\G_{**}}^{\otimes s} \otimes_{\F} \mbp_{**}(Y)).$$
Hence, the $E_2$-page has the usual description given in terms of $\Ext$-groups of left $\G_{**}$-comodules, i.e.
$$E_2^{s,t,u} \cong \Ext^{s,t,u}_{\G_{**}}(\mbp_{**}(X),\mbp_{**}(Y)).$$
By standard formal reasons, this spectral sequence actually converges to the groups $[\Sigma^{t-s,u}X,Y^{\wedge}_{\X \wedge \mbp}]$. We only have to notice that
$$Y^{\wedge}_{\X \wedge \mbp} \cong Y^{\wedge}_{\X \wedge \hz} \cong Y^{\wedge}_{\hz}.$$
The second isomorphism comes from the same argument of the proof of Proposition \ref{conv}. Regarding the first isomorphism, we may consider following \cite[Section 7.3]{DI} the bicompletion $Y^{\wedge}_{\{\X \wedge \mbp,\X \wedge \hz\}}$. This spectrum may be obtained by computing the homotopy limit of the following cosimplicial spectrum
$$\xymatrix{(\X \wedge \hz \wedge Y)^{\wedge}_{\X \wedge \mbp} \ar@<-.75ex>[r] \ar@<.75ex>[r] & ((\X \wedge \hz)^{\wedge 2} \wedge Y)^{\wedge}_{\X \wedge \mbp} \ar@<-.75ex>[r] \ar@<0ex>[r] \ar@<.75ex>[r] & ((\X \wedge \hz)^{\wedge 3} \wedge Y)^{\wedge}_{\X \wedge \mbp} \ar@<-.75ex>[r] \ar@<-.25ex>[r] \ar@<.25ex>[r] \ar@<.75ex>[r] & \dots}
$$
or, equivalently, by computing the homotopy limit of the following cosimplicial spectrum 
$$\xymatrix{(\X \wedge \mbp \wedge Y)^{\wedge}_{\X \wedge \hz} \ar@<-.75ex>[r] \ar@<.75ex>[r] & ((\X \wedge \mbp)^{\wedge 2} \wedge Y)^{\wedge}_{\X \wedge \hz} \ar@<-.75ex>[r] \ar@<0ex>[r] \ar@<.75ex>[r] & ((\X \wedge \mbp)^{\wedge 3} \wedge Y)^{\wedge}_{\X \wedge \hz} \ar@<-.75ex>[r] \ar@<-.25ex>[r] \ar@<.25ex>[r] \ar@<.75ex>[r] & \dots.}
$$
Since $\hz$ is a motivic $\mbp$-module we have that for any $n$
$$((\X \wedge \hz)^{\wedge n} \wedge Y)^{\wedge}_{\X \wedge \mbp} \cong (\X \wedge \hz)^{\wedge n} \wedge Y$$
from which it follows that the first homotopy limit is just $Y^{\wedge}_{\X \wedge \hz}$. On the other hand, we know that $\X \wedge \mbp$ is $\hz$-complete, thus we get that for any $n$
$$((\X \wedge \mbp)^{\wedge n} \wedge Y)^{\wedge}_{\X \wedge \hz} \cong (\X \wedge \mbp)^{\wedge n} \wedge Y$$ 
and the second homotopy limit gives back $Y^{\wedge}_{\X \wedge \mbp}$. This implies that $Y^{\wedge}_{\X \wedge \mbp} \cong Y^{\wedge}_{\X \wedge \hz}$.

It only remains to prove the strong convergence. The arguments are the same as in \cite[Theorem 3.2]{GWX} and we report them here only for completeness. First, suppose that $\mbp_{**}(X)$ is concentrated in Chow-Novikov degrees $[a,b]$ and $\mbp_{**}(Y)$ is concentrated in Chow-Novikov degrees $[c,d]$. Therefore, the $E_1$-page and so all the following pages are trivial outside the range $c-b+2u \leq t \leq d-a+2u$. Now, note that the differential on the $E_r$-page has, as usual, the tridegree $(r,r-1,0)$, which means in particular that it is trivial when $r-1 > d-a-c+b$. This amounts to say that the spectral sequence collapses at the $E_{d-a-c+b+2}$-page, and so it is strongly convergent, which completes the proof. 
\end{proof}

\begin{dfn}
	\normalfont
Let $\X-\mo_{cell,\hz}$ be the full triangulated subcategory of $\X-\mo_{cell}$ consisting of $\hz$-complete cellular isotropic spectra. Denote by $\X-\mo^{b,\geq 0}_{cell,\hz}$ the full subcategory of $\X-\mo^{b}_{cell,\hz}$ whose objects have $\mbp$-homology concentrated in non-negative Chow-Novikov degrees and by $\X-\mo^{b,\leq 0}_{cell,\hz}$ the full subcategory of $\X-\mo^{b}_{cell,\hz}$ whose objects have $\mbp$-homology concentrated in non-positive Chow-Novikov degrees. Finally, let $\X-\mo^{\heartsuit}_{cell,\hz}$ be the full subcategory whose objects are both in $\X-\mo^{b,\geq 0}_{cell,\hz}$ and in $\X-\mo^{b,\leq 0}_{cell,\hz}$, i.e. have $\mbp$-homology concentrated in Chow-Novikov degree $0$.
\end{dfn} 

We want to point out that, since $\X \wedge \hz$ is a $\X \wedge \mbp$-module and $\X \wedge \mbp$ is $\X \wedge \hz$-complete, the subcategories of $\hz$-complete and $\mbp$-complete isotropic spectra coincide. 

The next corollary, which corresponds to \cite[Corollary 4.7]{GWX}, computes hom-sets from $\X-\mo^{b,\geq 0}_{cell,\hz}$ to $\X-\mo^{b,\leq 0}_{cell,\hz}$ in algebraic terms.

\begin{cor}\label{hom}
	Let $k$ be a flexible field, $X$ an object in $\X-\mo^{b,\geq 0}_{cell,\hz}$ and $Y$ in $\X-\mo^{b,\leq 0}_{cell,\hz}$. Then, the functor $\mbp_{**}$ provides an isomorphism
	$$[X,Y] \cong \Hom ^{0,0}_{\G_{**}}(\mbp_{**}(X),\mbp_{**}(Y)).$$
\end{cor}
\begin{proof}
	As we have already pointed out, the $E_1$-page of the isotropic Adams-Novikov spectral sequence is given by
	$$E_1^{s,t,u} \cong \Hom ^{t,u}_{\G_{**}}(\mbp_{**}(X),\G_{**} \otimes_{\F} \overline{\G_{**}}^{\otimes s} \otimes_{\F} \mbp_{**}(Y)).$$
	Since we are interested in the group $[X,Y]$, the part of the $E_1$-page that is involved consists of the groups in tridegrees $(t,t,0)$. By hypothesis, $X$ is in $\X-\mo^{b,\geq 0}_{cell,\hz}$ while $Y$ is in $\X-\mo^{b,\leq 0}_{cell,\hz}$, which implies that, among these groups, only $E_1^{0,0,0}$ is non-trivial. Since in this tridegree all differentials from the second on are trivial by degree reasons, we have that
	$$[X,Y] \cong E_2^{0,0,0} \cong \Ext^{0,0,0}_{\G_{**}}(\mbp_{**}(X),\mbp_{**}(Y)) \cong  \Hom ^{0,0}_{\G_{**}}(\mbp_{**}(X),\mbp_{**}(Y))$$
	which completes the proof.
\end{proof}

By using the isotropic Adams-Novikov spectral sequence we also get the following corollary that corresponds to \cite[Corollary 4.8]{GWX} and is a generalisation of \cite[Theorem 5.7]{T}.

\begin{cor}\label{morcor}
Let $k$ be a flexible field and $X$ and $Y$ objects in $\X-\mo^{\heartsuit}_{cell,\hz}$. Then, there is an isomorphism
$$[\Sigma^{t,u}X,Y] \cong \Ext^{2u-t,2u,u}_{\G_{**}}(\mbp_{**}(X),\mbp_{**}(Y)).$$
\end{cor}
\begin{proof}
This follows immediately by noticing that the differentials $d_r^{s,t,u}:E_r^{s,t,u} \rightarrow E_r^{s+r,t+r-1,u}$ of the isotropic Adams-Novikov spectral sequence are trivial for $r \geq 2$ since $E_2^{s,t,u}$ is trivial for $t \neq 2u$. Hence, the spectral sequence is strongly convergent
and collapses at the second page, from which we get that
$$[\Sigma^{t,u}X,Y] \cong E_2^{2u-t,2u,u} \cong \Ext^{2u-t,2u,u}_{\G_{**}}(\mbp_{**}(X),\mbp_{**}(Y))$$
that is what we wanted to show. 
\end{proof}

Before proceeding, we also need the following lemma which essentially corresponds to \cite[Lemma 4.10]{GWX}.

\begin{lem}\label{ext}
Let $k$ be a flexible field and $M$ a $\G_{**}$-comodule concentrated in Chow-Novikov degree 0 which is finitely generated as an $\F$-vector space. Then, there exists an object $X$ in $\X-\mo^{\heartsuit}_{cell,\hz}$ such that $M \cong \mbp_{**}(X)$.
\end{lem}
\begin{proof}
Since by hypothesis $M$ is a finite-dimensional $\F$-vector space, by \cite[Theorem 3.3]{L} one has a finite filtration of subcomodules
$$0 \cong M_0 \subset M_1 \subset \dots \subset M_n \cong M$$
such that, for any $i$, the quotient $M_i/M_{i-1}$ is stably isomorphic to $\F$, i.e. $M_i/M_{i-1} \cong \Sigma^{2q_i,q_i} \F$ for some integer $q_i$. We want to prove the statement by induction on $i$. First, note that by Theorem \ref{hgmbp} the comodule $\Sigma^{2q_i,q_i} \F$ is the $\mbp$-homology of the isotropic spectrum $\Sigma^{2q_i,q_i} \X^{\wedge}_{\hz}$ for any $i$. Now, suppose that there exists an object $X_{i-1}$ in $\X-\mo^{\heartsuit}_{cell,\hz}$ such that $M_{i-1} \cong \mbp_{**}(X_{i-1})$. Then, the short exact sequence
$$0 \rightarrow M_{i-1} \rightarrow M_i
 \rightarrow \Sigma^{2q_i,q_i} \F \rightarrow 0$$
represents an element of $\Ext_{\G_{**}}^{1,0,0}(\Sigma^{2q_i,q_i} \F,M_{i-1})$, namely a morphism $f_i$ in $[\Sigma^{2q_i-1,q_i}\X^{\wedge}_{\hz},X_{i-1}]$ by Corollary \ref{morcor}. Let us define $X_i$ as $Cone(f_i)$. Then, we have a long exact sequence in $\mbp$-homology
$$\dots \rightarrow \Sigma^{2q_i-1,q_i} \F \xrightarrow{0} M_{i-1} \rightarrow \mbp_{**}(X_i) \rightarrow \Sigma^{2q_i,q_i} \F \xrightarrow{0} \Sigma^{1,0}M_{i-1} \rightarrow \dots .$$
Note that the connecting homomorphism 
$$g_{i*}:\Ext^{0,0,0}( \Sigma^{2q_i,q_i} \F, \Sigma^{2q_i,q_i} \F) \rightarrow \Ext^{1,0,0}( \Sigma^{2q_i,q_i} \F,M_{i-1})$$
described as the Yoneda product with the element $g_i$ of $\Ext_{\G_{**}}^{1,0,0}(\Sigma^{2q_i,q_i} \F,M_{i-1})$
corresponding to the short exact sequence
$$0 \rightarrow M_{i-1} \rightarrow \mbp_{**}(X_i) \rightarrow \Sigma^{2q_i,q_i} \F \rightarrow 0$$
converges to the map
$$f_{i*}: [\Sigma^{2q_i-1,q_i}\X^{\wedge}_{\hz} ,\Sigma^{2q_i-1,q_i}\X^{\wedge}_{\hz} ] \rightarrow [\Sigma^{2q_i-1,q_i}\X^{\wedge}_{\hz},X_{i-1}]$$
induced by $f_i$ in isotropic homotopy groups (see \cite[Theorem 2.3.4]{Ra}). By Corollary \ref{morcor} the isotropic Adams-Novikov spectral sequence collapses at the second page, so $g_{i*}=f_{i*}$. It follows that the extensions $g_i$ and $f_i$ coincide which implies that $\mbp_{**}(X_i) \cong M_i$, as we wanted to prove.
\end{proof}

The next result is the isotropic equivalent of \cite[Lemma 4.2]{GWX}.

\begin{lem}
Let $k$ be a flexible field and $X_{\alpha}$ be a filtered system in $\X-\mo^{\heartsuit}_{cell,\hz}$. Then, also the colimit ${\mathrm colim} \: X_{\alpha}$ in $\X-\mo_{cell}$ belongs to $\X-\mo^{\heartsuit}_{cell,\hz}$.	
	\end{lem}
\begin{proof}
First, note that, since $\mbp_{**}({\mathrm colim} \: X_{\alpha}) \cong \mathrm{colim} \: \mbp_{**}(X_{\alpha})$, ${\mathrm colim} \: X_{\alpha}$ has $\mbp$-homology concentrated in Chow-Novikov degree 0. Moreover, recall from \cite[Corollary A1.2.12]{Ra} that $\Ext_{\G_{**}}(\F,-)$ may be computed as the homology of the cobar complex for the second variable. Since the cobar complex preserves filtered colimits, so does $\Ext_{\G_{**}}(\F,-)$. Then, Corollary \ref{morcor} implies that
\begin{align*}
	\pi_{t,u}({\mathrm colim} \: X_{\alpha}) &\cong {\mathrm colim} \: \pi_{t,u}(X_{\alpha})\\
	& \cong {\mathrm colim}\: \Ext^{2u-t,2u,u}_{\G_{**}}(\F,\mbp_{**}(X_{\alpha}))\\
	& \cong \Ext^{2u-t,2u,u}_{\G_{**}}(\F,{\mathrm colim} \: \mbp_{**}(X_{\alpha}))\\
	& \cong \Ext^{2u-t,2u,u}_{\G_{**}}(\F, \mbp_{**}({\mathrm colim} \: X_{\alpha}))\\
	& \cong \pi_{t,u}(({\mathrm colim} \: X_{\alpha})^{\wedge}_{\hz}) 
	\end{align*}
from which it follows that ${\mathrm colim} \: X_{\alpha}$ is $\hz$-complete that concludes the proof. 
\end{proof}

We are now ready to identify $\X-\mo^{\heartsuit}_{cell,\hz}$ with the abelian category of left $\G_{**}$-comodules concentrated in Chow-Novikov degree 0 that we denote by $\G_{**}-\com^0_{**}$. The following proposition is an isotropic version of \cite[Proposition 4.11]{GWX}.

\begin{prop}\label{heart}
	Let $k$ be a flexible field. Then, the functor 
	$$\mbp_{**}:\X-\mo^{\heartsuit}_{cell,\hz} \xrightarrow{\cong} \G_{**}-\com^0_{**}$$
	is an equivalence of categories.
\end{prop}
\begin{proof}
	First, note that Corollary \ref{hom} guarantees that the functor $\mbp_{**}$ is fully faithful. We just need to show that it is essentially surjective. Recall from \cite[Propositions 1.4.10, 1.4.4 and 1.4.1]{Ho} that any left $\G_{**}$-comodule $M$ is a filtered colimit of comodules $M_{\alpha}$ which are finitely generated as $\F$-vector spaces. By Lemma \ref{ext} all $M_{\alpha}$ are expressible as $\mbp_{**}(X_{\alpha})$ for some $X_{\alpha}$ in $\X-\mo^{\heartsuit}_{cell,\hz}$. Therefore, we have that $M \cong \mbp_{**}(X)$ where $X={\mathrm colim} \: X_{\alpha}$, which is what we aimed to show. 
\end{proof}

\begin{rem}
	\normalfont
	Note that $\G_{**}-\com^0_{**}$ is equivalent to the category of left $\A_{*}$-comodules, where $\A_{*}$ is the dual of the topological Steenrod algebra. Hence, the previous result can be rephrased by saying that $\X-\mo^{\heartsuit}_{cell,\hz}$ is equivalent to the abelian category of left $\A_*$-comodules.
\end{rem}

The next proposition that corresponds to \cite[Proposition 4.12]{GWX} provides $\X-\mo^b_{cell,\hz}$ with a $t$-structure.

\begin{prop}\label{tri}
Let $k$ be a flexible field. Then, the pair $(\X-\mo^{b,\geq 0}_{cell,\hz},\X-\mo^{b,\leq 0}_{cell,\hz})$ defines a bounded $t$-structure on $\X-\mo^b_{cell,\hz}$.
\end{prop}
\begin{proof}
	Just by the definition of $\X-\mo^{b,\geq 0}_{cell,\hz}$ and $\X-\mo^{b,\leq 0}_{cell,\hz}$ we know that the first is closed under suspensions, the second under desuspensions and both under extensions. Moreover, clearly 
	$$\X-\mo^{b}_{cell,\hz} = \bigcup_{n \in \Z} \X-\mo^{b,\geq n}_{cell,\hz}$$
	where $\X-\mo^{b,\geq n}_{cell,\hz}$ is the $n$-th suspension of $\X-\mo^{b,\geq 0}_{cell,\hz}$. Now, consider an object $X$ in $\X-\mo^{b,\geq 0}_{cell,\hz}$ and an object $Y$ in $\X-\mo^{b,\leq -1}_{cell,\hz}$, i.e. the first desuspension of $\X-\mo^{b,\leq 0}_{cell,\hz}$. Then, by Corollary \ref{hom}
	$$[X,Y] \cong \Hom ^{0,0}_{\G_{**}}(\mbp_{**}(X),\mbp_{**}(Y)) \cong 0$$
	since $\mbp_{**}(X)$ is concentrated in non-negative Chow-Novikov degrees while $\mbp_{**}(Y)$ is concentrated in negative Chow-Novikov degrees. Finally, let $X$ be an object in $\X-\mo^{b,\geq 0}_{cell,\hz}$, then $\mbp(X)$ is concentrated in non-negative Chow-Novikov degrees. Consider the projection $\mbp(X) \rightarrow \mbp(X)_0$ that kills all the elements in positive Chow-Novikov degrees, and note that there exists an object $X_0$ in $\X-\mo^{\heartsuit}_{cell,\hz}$ such that $\mbp(X_0) \cong \mbp(X)_0$. Now, by Corollary \ref{hom} this morphism comes from a map $f:X \rightarrow X_0$ such that $\Sigma^{-1,0}Cone(f)$ belongs to $\X-\mo^{b,\geq 1}_{cell,\hz}$. Therefore, by \cite[Proposition 3.6]{GWX}, the pair $(\X-\mo^{b,\geq 0}_{cell,\hz},\X-\mo^{b,\leq 0}_{cell,\hz})$ defines a bounded $t$-structure on $\X-\mo^b_{cell,\hz}$ that is what we aimed to prove. 
\end{proof}

We are now ready to prove the main result of this section that corresponds to \cite[Theorem 4.13]{GWX}. In this theorem we identify $\X-\mo^b_{cell,\hz}$ with the derived category of left $\G_{**}$-comodules concentrated in Chow-Novikov degree 0.

\begin{thm}\label{main}
	Let $k$ be a flexible field. Then, there exists a $t$-exact equivalence of stable $\infty$-categories
	$$\D^b(\G_{**}-\com^0_{**}) \xrightarrow{\cong} \X-\mo^b_{cell,\hz}.$$
\end{thm}
\begin{proof}
	First, note that, by Propositions \ref{heart} and \ref{tri}, $(\X-\mo^{b,\geq 0}_{cell,\hz},\X-\mo^{b,\leq 0}_{cell,\hz})$ defines a bounded $t$-structure on $\X-\mo^b_{cell,\hz}$ whose heart is equivalent to the category of left $\G_{**}$-comodules concentrated in Chow-Novikov degree 0, so has enough injectives. Now, let $X$ and $Y$ be objects in $\X-\mo^{\heartsuit}_{cell,\hz}$ such that $\mbp_{**}(Y)$ is an injective $\G_{**}$-comodule. Then, in this case the isotropic Adams-Novikov spectral sequence 
	$$E_2^{s,t,u} \cong \Ext^{s,t,u}_{\G_{**}}(\mbp_{**}(X),\mbp_{**}(Y)) \Longrightarrow [\Sigma^{t-s,u}X,Y]$$
	collapses at the second page since the $E_2$-page is trivial for $s \neq 0$. Hence, we have that
	$$[\Sigma^{-i}X,Y] \cong \Ext^{0,-i,0}_{\G_{**}}(\mbp_{**}(X),\mbp_{**}(Y)) \cong \Hom ^{-i,0}_{\G_{**}}(\mbp_{**}(X),\mbp_{**}(Y)) \cong 0$$
for any $i>0$ since both $\mbp_{**}(X)$ and $\mbp_{**}(Y)$ are concentrated in Chow-Novikov degree $0$. It follows by \cite[Proposition 2.12]{GWX}, which is based on Lurie's recognition criterion \cite[Proposition 1.3.3.7]{Lu}, that there exists a $t$-exact equivalence of stable $\infty$-categories
$$\D^b(\G_{**}-\com^0_{**}) \xrightarrow{\cong} \X-\mo^b_{cell,\hz}$$
extending the equivalence on the hearts, which completes the proof.	  
\end{proof}

\begin{rem}
	\normalfont
	We point out that, given the identification $\G_{**} \cong \A_*$, the last theorem identifies as triangulated categories the category of bounded isotropic $\hz$-complete cellular spectra with the derived category of left $\A_*$-comodules, namely $\D^b(\A_{*}-\com_*).$
\end{rem}

By using the same argument as in \cite[Corollary 1.2]{GWX} one is able to obtain an unbounded version of the previous theorem identifying the whole $\X^{\wedge}_{\hz}-\mo_{cell}$ with Hovey's unbounded derived category $\st(\G_{**}-\com^0_{**})$, which is the same as  $\st(\A_{*}-\com_{*})$ (see \cite[Section 6]{Ho}).

\begin{cor}
	Let $k$ be a flexible field. Then, there exists an equivalence of stable $\infty$-categories
	$$ \X^{\wedge}_{\hz}-\mo_{cell} \cong \st(\G_{**}-\com^0_{**}).$$
\end{cor}

\section{The category of isotropic Tate motives}

We finish in this section by applying previous results in order to obtain information on the category of isotropic Tate motives $\DM(k/k)_{Tate}$. In particular, we get an easy algebraic description for the hom-sets in $\DM(k/k)_{Tate}$ between motives of isotropic cellular spectra.

First, we prove the following lemma which tells us that the isotropic motivic homology of an isotropic spectrum is always a free $H_{**}(k/k)$-module.

\begin{lem}\label{hiso}
	Let $k$ be a flexible field and $X$ an object in $\X-\mo$. Then, there exists an isomorphism of left $H_{**}(k/k)$-modules
	$$H^{iso}_{**}(X) \cong H_{**}(k/k) \otimes_{\F} \mbp_{**}(X).$$
\end{lem}
\begin{proof}
	The Hopkins-Morel equivalence (see \cite[Theorem 7.12]{H}) implies in particular that $\hz$ is a quotient spectrum of $\mbp$. It follows that $\hz$ can be obtained from $\mbp$ by applying cones and homotopy colimits and, so, it is an $\mbp$-cellular module from which we get by Theorem \ref{mbpcell} that $$\X \wedge \hz \cong \bigvee_{\alpha \in A} \Sigma^{p_{\alpha},q_{\alpha}}(\X \wedge \mbp)$$ for some set $A$. Now, note that by Theorem \ref{hgmbp}
	\begin{align*}
	H_{**}(k/k)& \cong \pi_{**}(\X \wedge \hz)\\
	& \cong \pi_{**}(\bigvee_{\alpha \in A}\Sigma^{p_{\alpha},q_{\alpha}}(\X \wedge \mbp))\\
	& \cong \bigoplus_{\alpha \in A}\Sigma^{p_{\alpha},q_{\alpha}}\pi_{**}(\X \wedge \mbp)\\
	& \cong \bigoplus_{\alpha \in A} \Sigma^{p_{\alpha},q_{\alpha}} \F.
	\end{align*}
At this point, let $X$ be an object in $\X-\mo$. Then, 
	\begin{align*}
	H^{iso}_{**}(X) &\cong \pi_{**}(\X \wedge \hz \wedge X)\\
	& \cong \pi_{**}(\bigvee_{\alpha \in A}\Sigma^{p_{\alpha},q_{\alpha}}(\X \wedge \mbp \wedge X))\\
	& \cong \bigoplus_{\alpha \in A}\Sigma^{p_{\alpha},q_{\alpha}}\pi_{**}(\X \wedge \mbp \wedge X) \\
	& \cong \bigoplus_{\alpha \in A} \Sigma^{p_{\alpha},q_{\alpha}} \mbp_{**}(X) \\
	&\cong H_{**}(k/k) \otimes_{\F} \mbp_{**}(X)
	\end{align*}
	that finishes the proof.	
\end{proof}

 In the next proposition we compute hom-sets in the isotropic triangulated category of motives between motives of isotropic cellular spectra. They happen to be isomorphic to hom-sets of left $H_{**}(k/k)$-modules between the respective isotropic homology.
 
\begin{prop}
	Let $k$ be a flexible field and $X$ and $Y$ objects in $\X-\mo_{cell}$. Then, there exists an isomorphism
	$$\Hom _{\DM(k/k)_{Tate}}({\mathrm M}(X),{\mathrm M}(Y)) \cong \Hom _{H_{**}(k/k)}(H^{iso}_{**}(X), H^{iso}_{**}(Y)).$$
\end{prop}
\begin{proof}
Consider the functor 
$$H^{iso}_{**}: \DM(k/k)_{Tate} \rightarrow H_{**}(k/k)-\mo_{**}$$
which sends each isotropic Tate motive to the respective isotropic motivic homology and let $X$ and $Y$ be motivic spectra in  $\X-\mo_{cell}$. Then, by Theorem \ref{mbpcell}, Lemma \ref{hiso} and \cite[Proposition 2.4]{T} we have that
\begin{align*}
	\Hom _{\DM(k/k)_{Tate}}({\mathrm M}(X),{\mathrm M}(Y)) &\cong [X,\X \wedge \hz \wedge Y] \\
	&\cong [\X \wedge \mbp \wedge X,\X \wedge \hz \wedge Y]_{\X \wedge \mbp} \\
	&\cong \Hom _{\F}(\pi_{**}(\X \wedge \mbp \wedge X), \pi_{**}(\X \wedge \hz \wedge Y))\\
	&\cong \Hom _{\F}(\mbp_{**}(X), H_{**}^{iso}(Y))\\
	&\cong \Hom _{H_{**}(k/k)}(H_{**}(k/k) \otimes_{\F} \mbp_{**}(X), H_{**}^{iso}(Y))\\
	&\cong \Hom _{H_{**}(k/k)}(H^{iso}_{**}(X), H^{iso}_{**}(Y))
	\end{align*}
which completes the proof. 
\end{proof}

\begin{rem}
	\normalfont
	The last result suggests that isotropic Tate motives coming from $\SH(k/k)_{cell}$ are very special in the sense that hom-sets in $\DM(k/k)_{Tate}$ between them are described simply in terms of hom-sets of free $H_{**}(k/k)$-modules. This property does not hold in general, so the next task should be to understand hom-sets in $\DM(k/k)_{Tate}$ between general isotropic Tate motives and try to describe them in algebraic terms. Unfortunately, since $H_{**}(k/k)$ is not concentrated in Chow-Novikov degree 0, the strategy used in \cite{GWX} and adapted in sections 7 and 8 of this paper does not immediately apply. Hence, some new ideas are needed and the hope is to develop them in following works.
\end{rem}

\footnotesize{
	}

\noindent {\scshape Mathematisches Institut, Ludwig-Maximilians-Universit\"at M\"unchen}\\
fabio.tanania@gmail.com


\begin{thebibliography}{00}

		
		\bibitem{BKWX} T. Bachmann, H. J. Kong, G. Wang, Z. Xu,
		\textit{The Chow t-structure on the $\infty$-category of motivic spectra}, Ann. of Math. (2) 195 (2022), no. 2, 707-773.
		
		\bibitem{Bo} A. K. Bousfield,
		\textit{The localization of spectra with respect to homology}, Topology, 18(4):257-281, 1979.
		
		\bibitem{DI2} D. Dugger, D. C. Isaksen,
		\textit{Motivic cell structures}. Algebraic \& Geometric Topology, 2005, 5.2: 615-652.
		
		\bibitem{DI} D. Dugger, D. C. Isaksen,
		\textit{The motivic Adams spectral sequence}, Geometry \& Topology, 14 (2010), no. 2, 967-1014. 
		
		\bibitem{G} B. Gheorghe,
		\textit{The motivic cofiber of $\tau$}, Doc. Math. 23, 1077-1127 (2018).
		
		\bibitem{GWX} B. Gheorghe, G. Wang, Z. Xu,
		\textit{The special fiber of the motivic deformation of the stable homotopy category is algebraic}, Acta Math. 226 (2021), no. 2, 319-407.
		
		\bibitem{Ho} M. Hovey, 
		\textit{Homotopy theory of comodules over a Hopf algebroid}, In: Homotopy Theory: Relations with Algebraic Geometry, Group Cohomology, and Algebraic $ K $-Theory. American Mathematical Soc., 2004. p. 261.
		
		\bibitem{H} M. Hoyois,
		\textit{From algebraic cobordism to motivic cohomology}, Journal f\"ur die reine und angewandte Mathematik, 2015, 2015.702: 173-226.
		
		\bibitem{I} D. C. Isaksen,
		\textit{Stable stems}, Mem. Amer. Math. Soc. 262 (2019), no. 1269, viii+159 pp.
		
		\bibitem{L} P. S. Landweber,
		\textit{Associated prime ideals and Hopf algebras}, Journal of pure and applied algebra, 1973, 3.1: 43-58.
		
		\bibitem{La} T. Lawson, 
		\textit{Secondary power operations and the Brown-Peterson spectrum at the prime 2}, Ann.
		Math. (2) 188 (2018), no. 2, 513-576.
		
		\bibitem{Le} M. Levine, 
		\textit{Aspects of enumerative geometry with quadratic forms}, Doc. Math. 25 (2020), 2179-2239.
		
		\bibitem{Lu} J. Lurie, 
		\textit{Higher algebra}, 
		http://people.math.harvard.edu/\textasciitilde lurie/papers/HA.pdf.
		
		\bibitem{Ma} L. Mantovani,
		\textit{Localizations and completions in motivic homotopy theory}, arXiv:1810.04134.
		
		\bibitem{Mo} F. Morel,
		\textit{On the motivic $\pi_0$ of the sphere spectrum}, (2004) In: Greenlees J.P.C. (eds) Axiomatic, Enriched and Motivic Homotopy Theory. NATO Science Series (Series II: Mathematics, Physics and Chemistry), vol 131. Springer, Dordrecht.
		
		\bibitem{MV} F. Morel, V. Voevodsky,
		\textit{$A^1$-homotopy theory of schemes}, Publications Math$\mathrm{\acute{e}}$matiques de l'I.H.$\mathrm{\acute{E}}$.S. no. 90.
		
		\bibitem{NSO} N. Naumann, M. Spitzweck, P. A. \O stv\ae r,
		\textit{Motivic Landweber Exactness}, Documenta Math. 14 (2009) 551-593. 
		
		\bibitem{Ra} D. C. Ravenel,
		\textit{Complex cobordism and stable homotopy groups of spheres}, American Mathematical Soc., 2003.
		
		\bibitem{T} F. Tanania,
		\textit{Isotropic stable motivic homotopy groups of spheres}, Adv. Math. 383 (2021), Paper No. 107696, 28 pp.
		
		\bibitem{Ve} G. Vezzosi,
		\textit{Brown-Peterson spectra in stable ${\mathbb A}^ 1$-homotopy theory}, Rendiconti del Seminario Matematico della Universit\`a di Padova, 2001, 106: 47-64.
		
		\bibitem{V} A. Vishik,
		\textit{Isotropic motives}, J. Inst. Math. Jussieu 21 (2022), no. 4, 1271-1330.
		
		\bibitem{V0} V. Voevodsky, 
		\textit{$A^1$-homotopy theory}, Proceedings of the International Congress of Mathematicians, Vol. I (Berlin, 1998). Doc. Math. 1998, Extra Vol. I, 579-604.
		
		\bibitem{V2} V. Voevodsky,
		\textit{Reduced power operations in motivic cohomology}, Publ. Math. Inst. Hautes $\mathrm{\acute{E}}$tudes Sci., pp. 1-57, 2003.
		
		\bibitem{V3} V. Voevodsky,
		\textit{Triangulated categories of motives over a field}, Cycles, transfers, and motivic homology theories, 188-238, Ann. of Math. Stud., 143, Princeton Univ. Press, Princeton, NJ, 2000.\\
		
\end{thebibliography}
\end{document}